\documentclass[11pt]{amsart}

\usepackage{amsfonts, amssymb, amsmath, graphicx}
\usepackage[all, knot]{xy}
\usepackage{multicol}
\usepackage{xcolor}

\makeatletter
\@namedef{subjclassname@2020}{%
  \textup{2020} Mathematics Subject Classification}
\makeatother

\newcommand{\R}{\mathbb{R}}
\newcommand{\Om}{\Omega}

\newtheorem{theorem}{Theorem}
\theoremstyle{plain}

\newtheorem{corollary}[theorem]{Corollary}
\newtheorem{lemma}{Lemma}

\newtheorem{remark}[theorem]{Remark}

\begin{document}

\title[Minimal generating sets of moves for isotopic diagrams]{Minimal generating sets of moves for diagrams of isotopic knots and spatial trivalent graphs}


\author{Carmen Caprau}
\address{Department of Mathematics, California State University, Fresno \linebreak 5245 North Backer Avenue, M/S PB108, CA 93740, USA}
\email{ccaprau@csufresno.edu}

\author{Bradley Scott}
\address{Department of Mathematics, California State University, Fresno\linebreak 5245 North Backer Avenue, M/S PB108, CA 93740, USA}
\email{bscott@gpusd.org}

\thanks{CC was partially supported by Simons Foundation grant $\#355640$}

\begin{abstract}
Polyak proved that all oriented versions of Reidemeister moves for knot and link diagrams can be generated by a set of just four oriented Reidemeister moves, and that no fewer than four oriented Reidemeister moves generate them all. We refer to a set containing four oriented Reidemeister moves that collectively generate all of the other oriented Reidemeister moves as a minimal generating set. Polyak also proved that a certain set containing two Reidemeister moves of type 1, one move of type 2, and one move of type 3 form a minimal generating set for all oriented Reidemeister moves. We expand upon Polyak's work by providing an additional eleven minimal, 4-element, generating sets of oriented Reidemeister moves, and we prove that these twelve sets represent all possible minimal generating sets of oriented Reidemeister moves.  We also consider the Reidemeister-type moves that relate oriented spatial trivalent graph diagrams with trivalent vertices that are sources and sinks and prove that a minimal generating set of oriented Reidemeister-type moves for spatial trivalent graph diagrams contains ten moves. 
\end{abstract}

\maketitle

\section{Introduction}

A \textit{link} with $n$ components is an equivalence class of smooth embeddings in $\mathbb{R}^3$ of $n$ copies of $S^1$, and a \textit{knot} is a one-component link. We will use the term `knot' to refer generically to both knots and links. A diagram of a knot $K$ is a generic projection of $K$ on $\mathbb{R}^2$ which contains transversal double points with overcrossing or undercrossing information.
A  well-known result~\cite{AB, Reid} states that two knot diagrams represent the same knot if and only if they are connected via planar isotopy and a finite sequence of \textit{Reidemeister moves} $\Om1, \Om2$ and $\Om3$ depicted in Figure~\ref{SpatialMoves}.
The Reidemeister moves~\cite{Reid} are local transformations on a knot diagram applied inside a disk in $\mathbb{R}^2$, and which keep the diagram unchanged outside of the disk. We refer to this disk as the \textit{localized disk}.

A \textit{knot invariant} is a quantity defined on the set of all knots which preserves the knot-type. In demonstrating that a certain quantity is a knot invariant, the quantity must be invariant under the moves $\Om1, \Om2$ and $\Om3$. Hence, Reidemeister moves play an important role when defining knot invariants.
When we consider oriented knots and their diagrams, we must work with oriented versions of the Reidemeister moves. There are four different versions of each of the moves $\Om1$ and $\Om2$, and eight different versions of the moves $\Om3$ (see Figures~\ref{Type1Moves},~
\ref{Type2Moves} and~\ref{Type3Moves}, respectively). With 16 oriented versions of Reidemeister moves, it can become tedious to check if a certain quantity defined on oriented knot diagrams yields an invariant for oriented knots. Hence, it is useful to have a generating set of oriented Reidemeister moves to minimize the work required to check for invariance.  A set of oriented Reidemeister moves $S$ is a \emph{generating set} if any oriented Reidemeister move $\Om$ may be realized by a finite sequence of planar isotopies and the moves in $S$ applied inside the localized disk of $\Om$.

Generating sets of oriented Reidemeister moves were studied by Polyak~\cite{Pol}, who proved that there exists a certain generating set of four moves containing two $\Om1$ moves, one $\Om2$ move, and one $\Om3$ move. We refer to the corresponding four moves, $\{\Om1a, \Om1b, \Om2a, \Om3a\}$,  as the \textit{Polyak moves}. He also proved that any generating set of oriented Reidemeister moves must contain at least two $\Om1$ moves. From here, we conclude that any generating set of oriented Reidemeister moves must contain at least two $\Om1$ moves and at least one move of type $\Om2$ and another of type $\Om3$. In this paper we refer to a set of four oriented Reidemeister moves that generate all of the other oriented Reidemeister moves as a \textit{minimal generating set}. The term `minimal' here refers to the fact that there are no sets of oriented Reidemeister moves of cardinality less than four that generate all of the 16 oriented Reidemeister moves. 

Polyak's result was long due, because in previous knot theory works there were inconsistencies on the number of moves required in a generating set of all oriented Reidemeister moves, and claims were provided without proofs.

In this paper we expand on Polyak's work and prove that there are exactly 12 minimal generating sets of oriented Reidemeister moves for knot diagrams. Another goal of this paper is to extend the notion of a generating set to include the moves for oriented spatial trivalent graph diagrams. A \emph{spatial trivalent graph} is an embedding in $\R^3$ of a graph with finitely many trivalent vertices. Every knot is a minor of a spatial trivalent graph, where we view a knot as a single vertex with a single loop edge. 
Thus the \textit{Reidemeister-type moves} for spatial trivalent graph diagrams include the (classical) Reidemeister moves $\Om1, \Om2$ and $\Om3$ together with the additional moves $\Om4$ and $\Om5$ involving trivalent vertices (see Figure~\ref{SpatialMoves}). Specifically, two diagrams of spatial trivalent graphs represent the same graph if and only if there is a finite sequence of the moves $\Om1$ through $\Om5$, along with planar isotopies, taking one diagram onto the other; a proof of this classical result can be found in~\cite[Section 2]{Ca}. For details on spatial graphs, we refer the reader to Kauffman's work in~\cite{KIG, KKP}. Just as for the Reidemeister moves, the moves $\Om4$ and $\Om5$ are local moves.

\begin{figure}[h]
\begin{center}
\includegraphics[scale=.2]{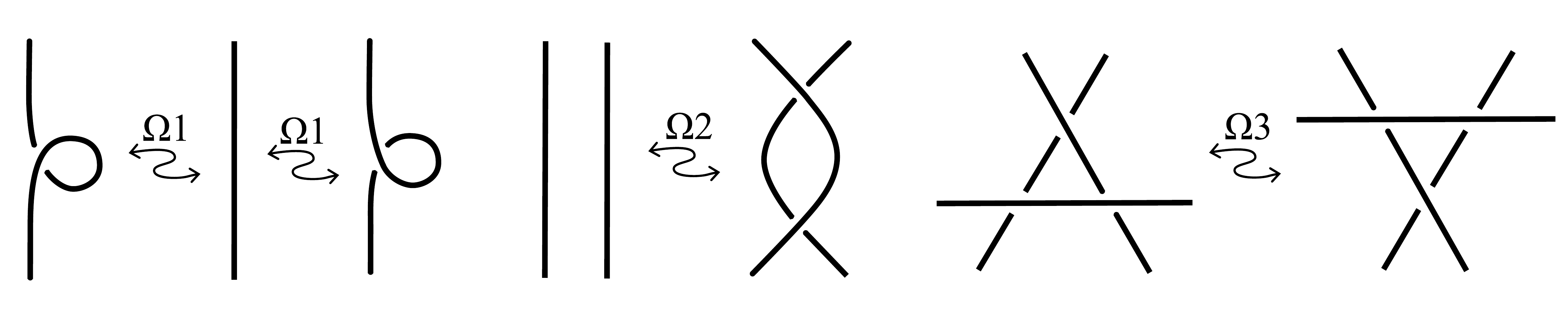}
\includegraphics[scale=.18]{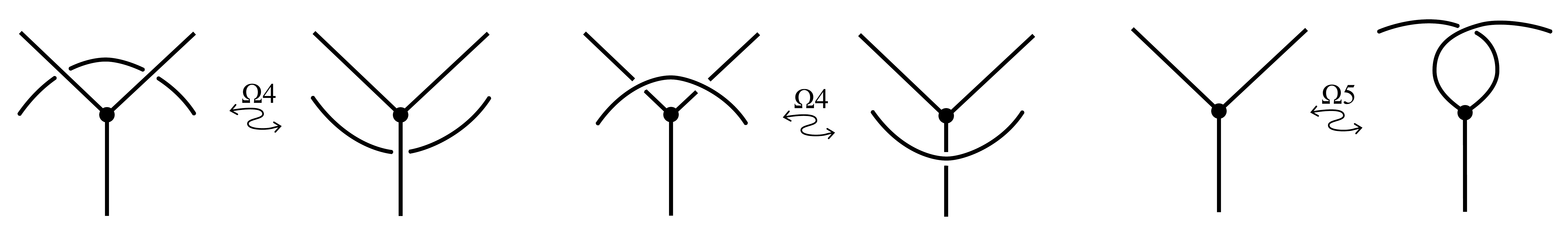}
\end{center}
\caption{Reidemeister-type moves for spatial trivalent graph diagrams}
\label{SpatialMoves}
\end{figure}

We consider oriented spatial trivalent graphs with trivalent vertices which are either sinks or sources (see Figure~\ref{sourcesink}). There are eight versions of oriented $\Om4$ moves and four versions of the oriented $\Om5$ moves (see Figures~\ref{Type4Moves} and ~\ref{Type5Moves}), and hence there are 28 versions of Reidemeister-type moves for oriented spatial trivalent graph diagrams with vertices that are sinks and sources. The second goal of this paper is to provide all minimal generating sets of oriented Reidemeister-type moves for spatial trivalent graph diagrams. We prove that such a set contains 10 moves (two $\Om1$ moves, one $\Om2$ move, one $\Om3$ move,  four $\Om4$ moves, and two $\Om5$ moves).

For the classical Reidemeister moves, when we distinguish between forward and backward moves, we obtain 32 different oriented Reidemeister moves, which Suwara~\cite{Su} called \textit{directed oriented Reidemeister moves}. In~\cite{Su}, Suwara proved that the eight directed Polyak moves form a minimal generating set of directed oriented Reidemeister moves. 
We note that a generating set of oriented Reidemeister-type moves for singular links was provided in~\cite{BEHY}. A singular link can also be regarded as an embedding in $\mathbb{R}^3$ of a 4-valent graph with rigid vertices; descriptions of rigid-vertex isotopies can be found in~\cite{KIG, KKP}.

The paper is organized as follows. In Section~\ref{sec:ReidMoves} we describe the 16 different oriented versions of the Reidemeister moves $\Om1, \Om2$ and $\Om3$, and how we differentiate them. In Section~\ref{sec:GenSets} we provide two collections $\mathcal{A}$ and $\mathcal{H}$, each of which contains six 4-element sets of Reidemeister moves. Each of the 4-element sets contains two $\Om1$ moves, one $\Om2$ move, and one $\Om3$ move, and we prove that each of the 12 sets in $\mathcal{A} \cup \mathcal{H}$ generates all of the oriented Reidemeister moves. The unique $\Om3$ move in each set in $\mathcal{A}$ is $\Om3a$, and the unique $\Om3$ move in each set in $\mathcal{H}$ is $\Om3h$. Moreover, the unique $\Om2$ move in any of the sets in $\mathcal{A} \cup \mathcal{H}$ is either $\Om2a$ or $\Om2b$.  In Section~\ref{sec:NonGenSets} we prove that the sets in collections $\mathcal{A}$ and $\mathcal{H}$ are minimal generating sets, and that they are all of the minimal generating sets of oriented Reidemeister moves. We extend these results in Section~\ref{sec:GeneratingSets-graphs}, where we find minimal generating sets of oriented Reidemeister-type moves for oriented spatial trivalent graph diagrams with vertices that are either sources or sinks.

\section{Minimal generating sets of Reidemeister moves} \label{generating sets}

We start by describing the oriented Reidemeister moves for knot diagrams and explaining how we differentiate between the different oriented versions of the moves. Then we provide twelve generating sets of oriented Reidemeister moves and prove that they indeed generate all of the moves. We also prove that these twelve sets are the only generating sets containing four oriented Reidemeister moves.

\subsection{The Oriented Reidemeister Moves} \label{sec:ReidMoves}

Considering all possible orientations, there are four type 1 moves, $\Om1$, four type 2 moves, $\Om2$, and eight type 3 moves, $\Om3$, which we label with the same conventions used by Polyak in~\cite{Pol}. 

The four oriented Reidemeister type 1 moves (shown in Figure~\ref{Type1Moves}) differ by the type of crossing introduced and the orientation of the resulting loop. The moves $\Om1a$ and $\Om1b$  involve a positive crossing and a clockwise and counter-clockwise oriented loop, respectively. Similarly, moves $\Om1c$ and $\Om1d$ involve a negative crossing and a clockwise and counter-clockwise oriented loop, respectively.
\begin{figure}[h]
\begin{center}
\includegraphics[scale=.22]{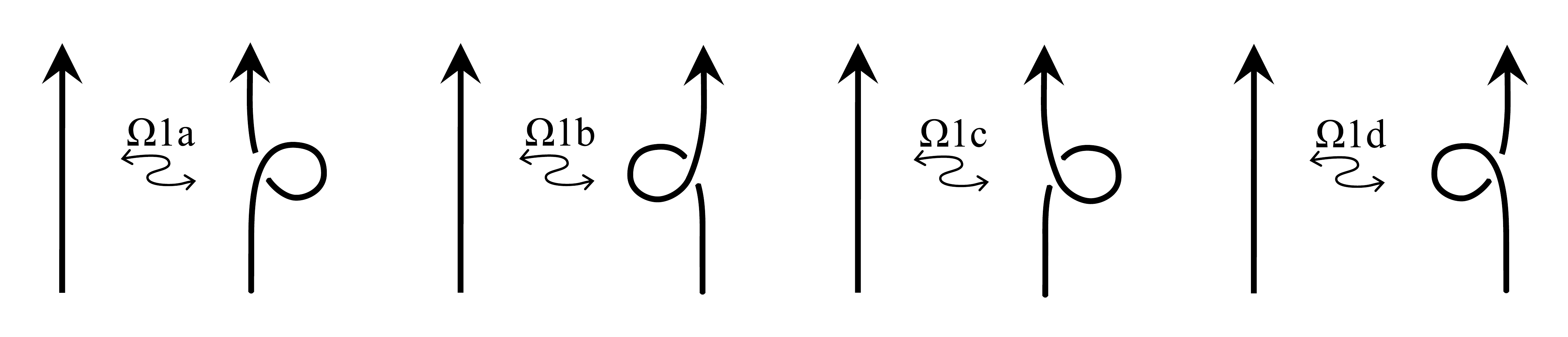}
\end{center}
\caption{Oriented Reidemeister moves of type 1}
\label{Type1Moves}
\end{figure}

The Reidemeister moves of type 2, depicted in Figure~\ref{Type2Moves},  introduce or remove two crossings in a knot diagram (depending on the direction of the move), which are transverse double points bounding a bigon in the underlying graph. The moves $\Om2a$ and $\Om2b$ introduce or remove a disoriented bigon, while the moves $\Om2c$ and $\Om2d$ introduce or remove a well-oriented bigon. The moves $\Om2a$ and $\Om2b$ are different because when following the orientations of the edges forming the disoriented bigon, in the move $\Om2a$ there is first a positive crossing followed by a negative crossing, while in the move $\Om2b$ we first see a negative crossing. The move $\Om2c$ introduces or removes a counter-clockwise oriented bigon, and the move $\Om2d$ introduces or removes a clockwise oriented bigon.
\begin{figure}[h]
\begin{center}
\includegraphics[scale=.2]{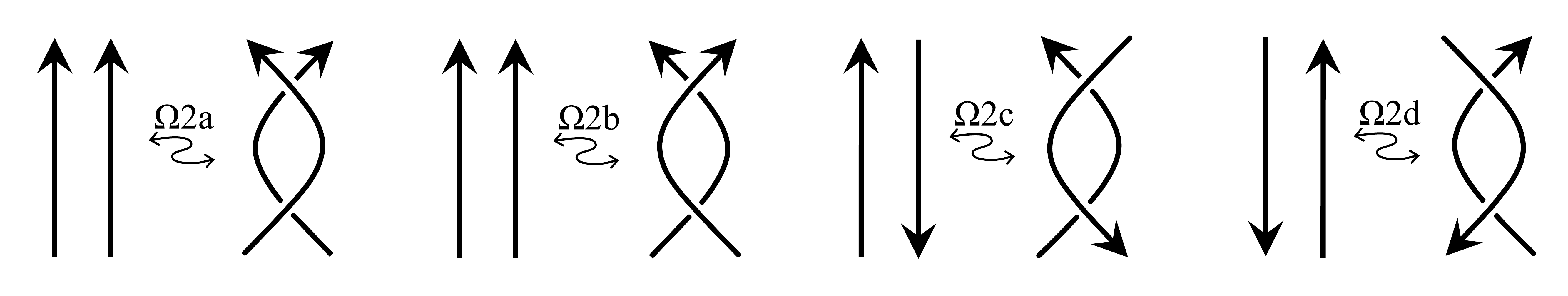}
\end{center}
\caption{Oriented Reidemeister moves of type 2}
\label{Type2Moves}
\end{figure}

The eight oriented Reidemeister moves of type 3, shown in Figure~\ref{Type3Moves}, are described by the number and position of positive and negative crossings involved in the move and the orientation of the enclosed region. Considering the underlying graph of the knot diagram in the localized disk of an oriented $\Om3$ move, a well-oriented or disoriented triangle is formed by the three edges joining the three transverse double points. 
The moves $\Om3a$, $\Om3d$ and $\Om3e$ each have two positive crossings and one negative crossing, but the move $\Om3a$ involves a well-oriented triangle while moves $\Om3d$ and $\Om3e$ involve disoriented triangles. In addition, in the move $\Om3d$ a strand slides over the negative crossing, but in the move $\Om3e$ a strand slides under the negative crossing.

In a similar fashion, we can distinguish the moves $\Om3c$, $\Om3f$ and $\Om3h$, which each involve two negative crossings and one positive crossing. The move $\Om3h$ applies to diagrams with a well-oriented triangle, while the moves $\Om3c$ and $\Om3f$ involve a disoriented triangle. In the move $\Om3c$ a strand slides over the positive crossing, and in the move $\Om3f$ a strand slides under the positive crossing.

 The remaining two moves, $\Om3b$ and $\Om3g$, involve three positive and negative crossings, respectively.
 \begin{figure}[h]
\begin{center}
\includegraphics[scale=.2]{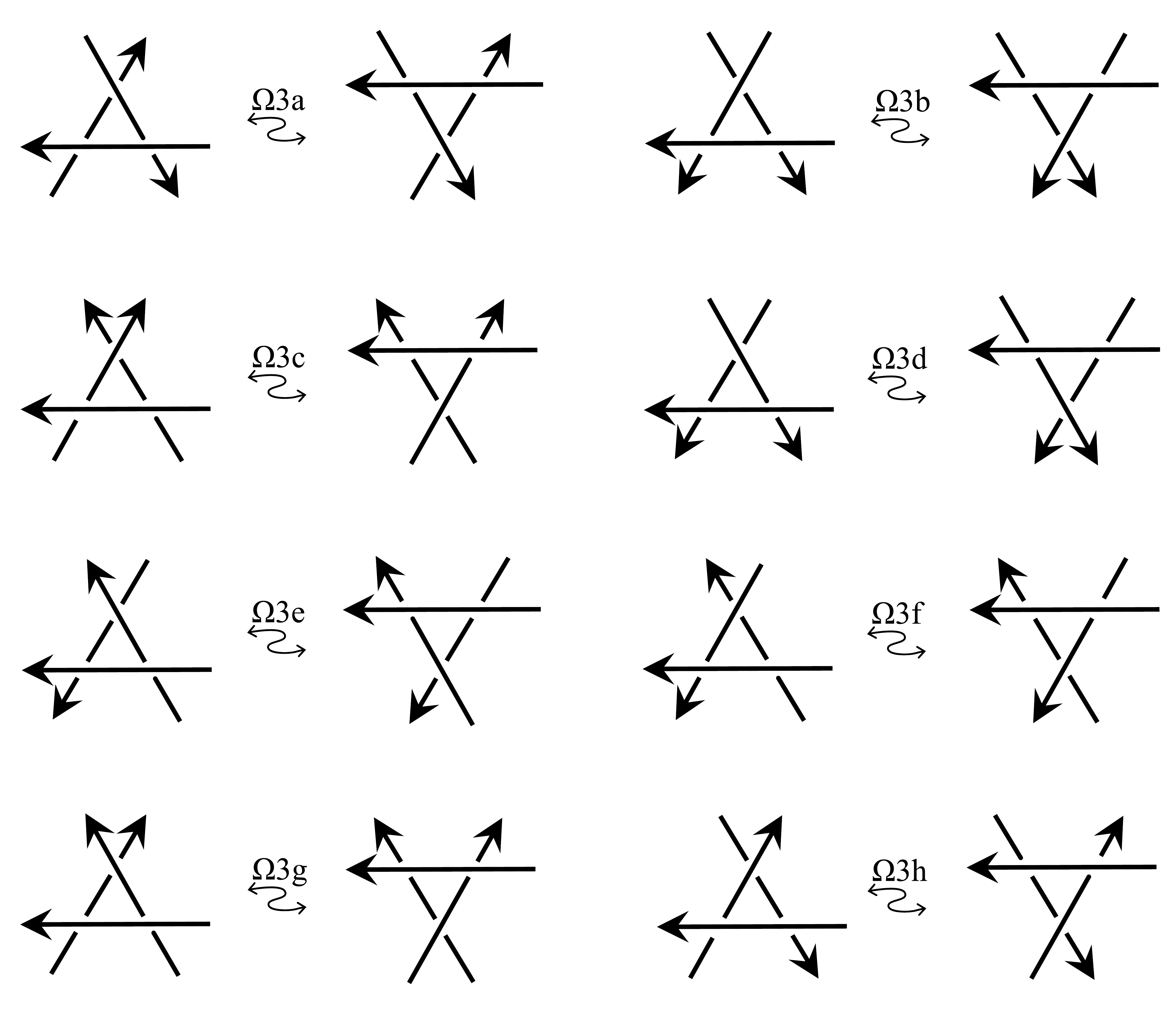}
\end{center}
\caption{Oriented Reidemeister moves of type 3}
\label{Type3Moves}
\end{figure}

\subsection{Generating the Oriented Reidemeister Moves} \label{sec:GenSets}

In~\cite{Pol}, Michael Polyak proves that a set of four specific oriented Reidemeister moves generates all of the other oriented Reidemeister moves for knot diagrams, and that this cannot be done with fewer than four moves.  We say that a generating set $S$ of oriented Reidemeister moves is minimal, if  there is no generating set $R$ of oriented Reidemeister moves such that $|R| < |S|$. In this section, we build upon Polyak's work by finding all of the minimal generating sets of oriented Reidemeister moves.

To begin, we present Polyak's minimal generating set below.

\begin{theorem}\cite[Theorem 1.1]{Pol}
Let $D$ and $D'$ be two diagrams in $\R^2$ representing the same oriented knot. Then one may pass from $D$ to $D'$ by isotopy and a finite sequence of four oriented Reidemeister moves $\Om1a, \Om 1b, \Om2a,$ and $\Om3a$.
\end{theorem}

Polyak proved that these four moves generate all other oriented Reidemeister moves.  In constructing the remaining minimal generating sets and proving that they indeed generate all of the other oriented Reidemeister moves, we borrow and adapt some of Polyak's ideas.

We define the following two collections of six sets, each of which containing four oriented Reidemeister moves:
\begin{align*}
\mathcal{A} = \{&\{\Om1a, \Om1c, \Om2a, \Om3a\}, \{\Om1a, \Om1c, \Om2b, \Om3a\}, \{\Om1b, \Om1d, \Om2a, \Om3a\},\\
 &\{\Om1b, \Om1d, \Om2b, \Om3a\}, \{\Om1a, \Om1b, \Om2a, \Om3a\}, \{\Om1a, \Om1b, \Om2b, \Om3a\}\}\\
\mathcal{H} = \{&\{\Om1a, \Om1c, \Om2a, \Om3h\}, \{\Om1a, \Om1c, \Om2b, \Om3h\}, \{\Om1b, \Om1d, \Om2a, \Om3h\},\\
 &\{\Om1b, \Om1d, \Om2b, \Om3h\}, \{\Om1c, \Om1d, \Om2a, \Om3h\}, \{\Om1c, \Om1d, \Om2b, \Om3h\}\}
\end{align*}
Note that for each set $X$ in $\mathcal{A}$ the move $\Om3a \in X$, and for each set $X$ in $\mathcal{H}$ the move $\Om3h \in X$. With these two collections of sets, we state the main result of this section.

\begin{theorem}\label{P1Main}
If $X \in \mathcal{A} \cup \mathcal{H}$, then $X$ is a generating set for all oriented Reidemeister moves for knot diagrams.
\end{theorem}

The proof of this theorem is given at the end of this section, making use of the following Lemmas~\ref{2c2d} through~\ref{A3moves}. We note that the sets in collections $\mathcal{A}$ and $\mathcal{H}$ do not contain the moves $\Om2c$ and $\Om2d$, and we  begin by showing in the next statement that any set $X \in \mathcal{A} \cup \mathcal{H}$ generates either the move $\Om2c$ or the move $\Om2d$.
The first realization given for each of the moves was provided in~\cite{Pol}.

\begin{lemma}\label{2c2d}
The move $\Om2c$ may be realized by each of the following oriented Reidemeister moves:
\begin{itemize}
\item $\Om1a$, $\Om3a$ and either $\Om2a$ or $\Om2b$
\item $\Om1c$, $\Om3h$ and either $\Om2a$ or $\Om2b$
\end{itemize}
The move $\Om2d$ may be realized by each of the following oriented Reidemeister moves:
\begin{itemize}
\item $\Om1b$, $\Om3a$ and either $\Om2a$ or $\Om2b$
\item $\Om1d$, $\Om3h$ and either $\Om2a$ or $\Om2b$
\end{itemize}
\end{lemma}

\begin{proof}
We realize the move $\Om2c$ in four ways, all shown below. 
\begin{center}
\includegraphics[height=0.85in]{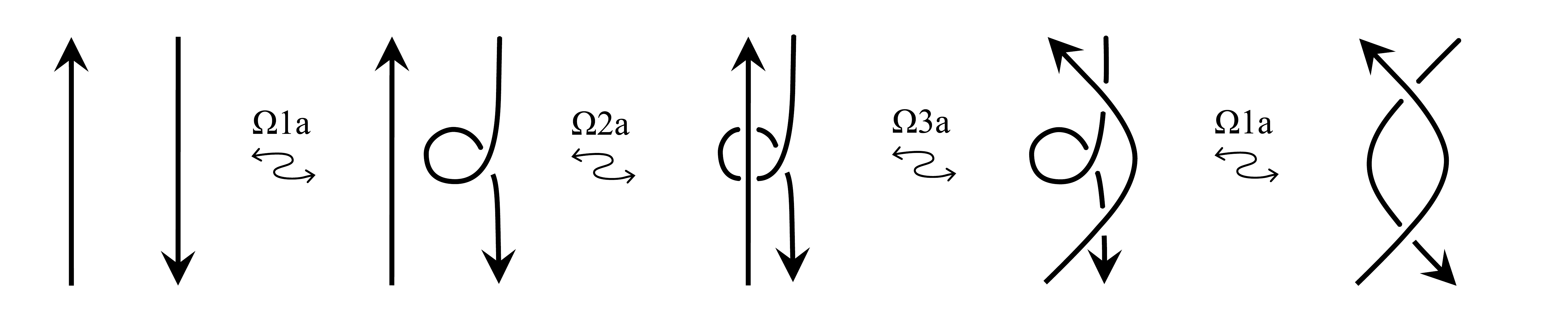}\\
\includegraphics[height=0.85in]{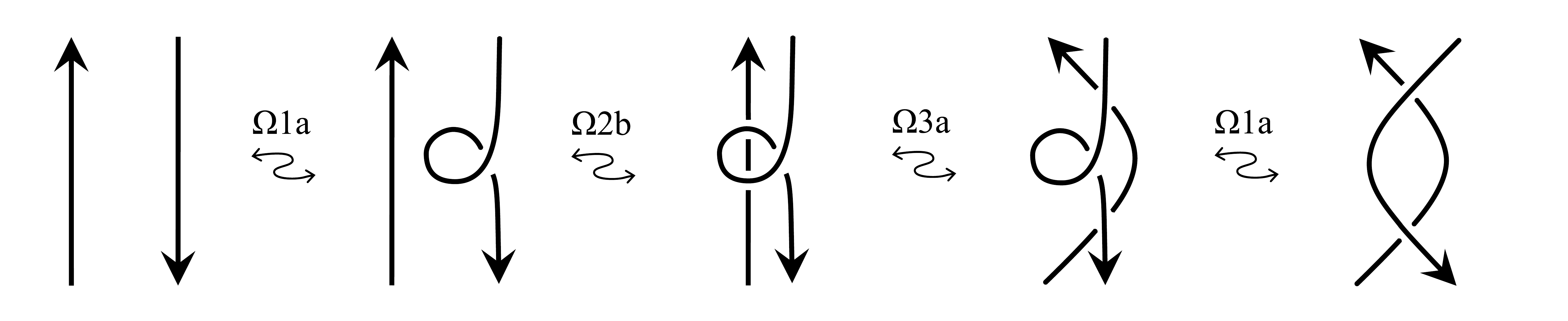}\\
\includegraphics[height=0.85in]{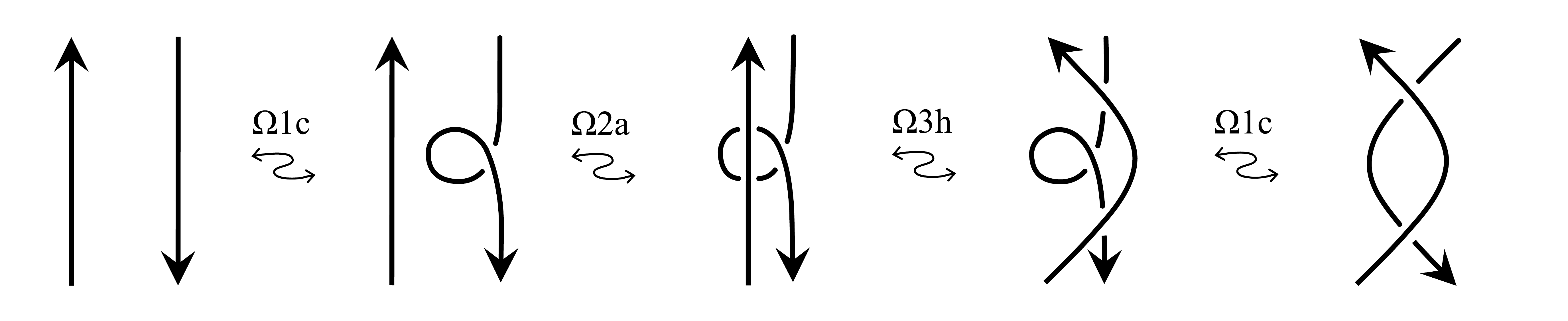}\\
\includegraphics[height=0.85in]{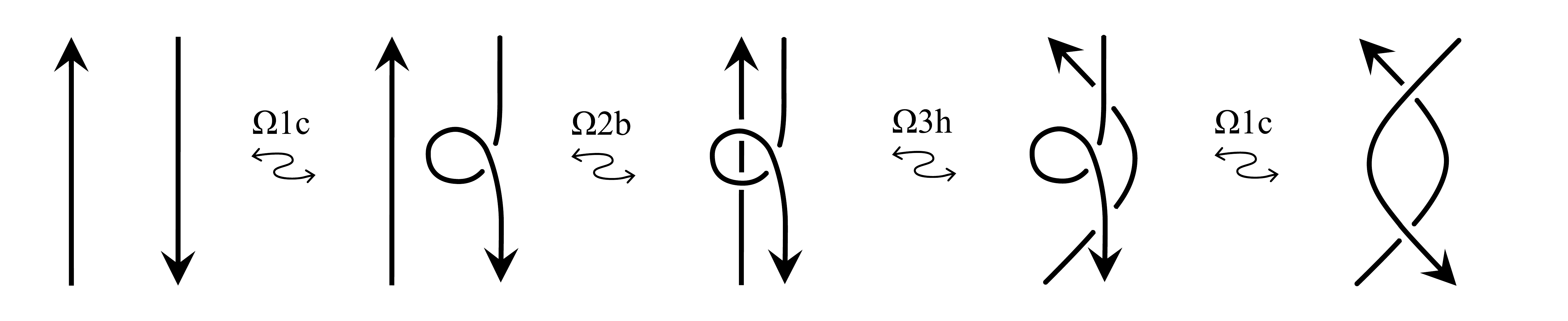}
\end{center}
The move $\Om2d$ is realized in a similar way, as shown below.
\begin{center}
\includegraphics[height=0.85in]{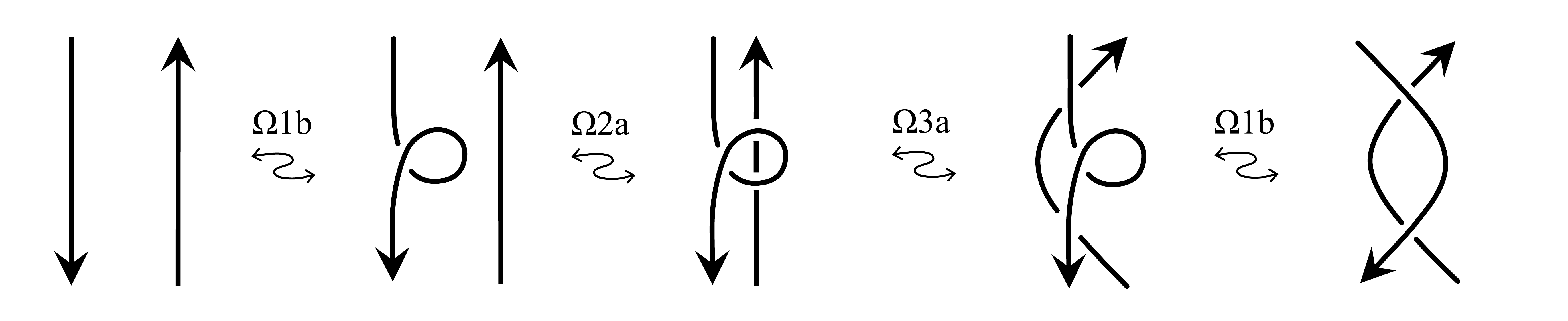}\\
\includegraphics[height=0.85in]{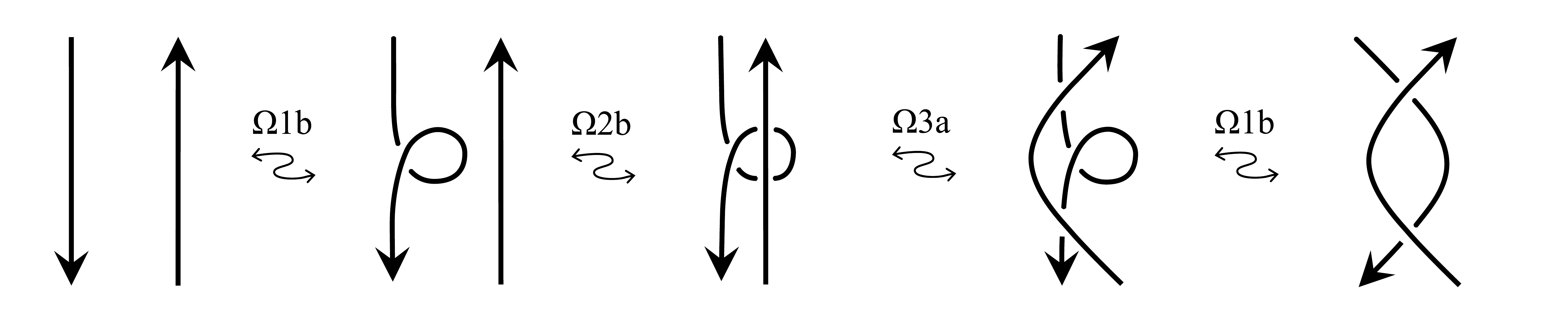}\\
\includegraphics[height=0.85in]{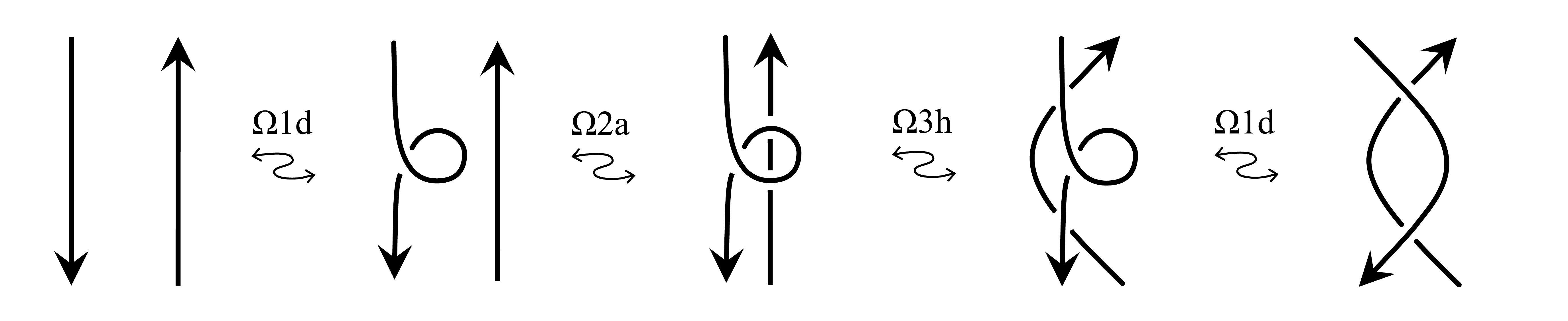}\\
\includegraphics[height=0.85in]{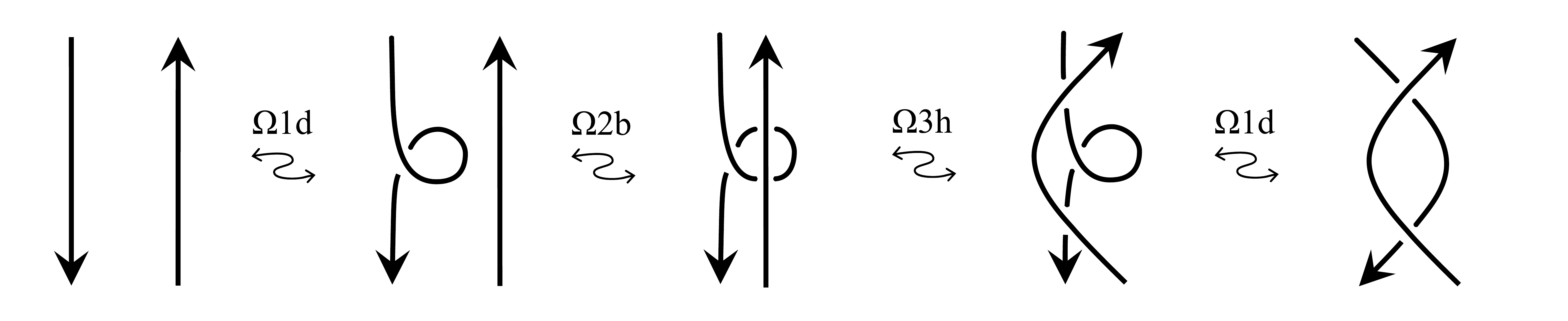}
\end{center}
Hence, the statement holds.
\end{proof}

\begin{remark}
With Lemma~\ref{2c2d}, each set from Theorem~\ref{P1Main} generates either move $\Om2c$ or move $\Om2d$. Hence, whenever the move $\Om2c$ or $\Om2d$ is used in a sequence of oriented Reidemeister moves, it may be replaced by a sequence of oriented Reidemeister moves from Lemma~\ref{2c2d}. 
\end{remark}

Each set of moves $X \in \mathcal{A} \cup \mathcal{H}$ contains exactly two of the $\Om1$ moves. In particular, any set $X \in \mathcal{A} \cup \mathcal{H}$ that generates the move $\Om2c$ with only sequences from Lemma~\ref{2c2d} contains move $\Om1a$ or $\Om1c$. Similarly, any set $X \in \mathcal{A} \cup \mathcal{H}$ that generates the move $\Om2d$ with only sequences from Lemma~\ref{2c2d} contains move $\Om1b$ or $\Om1d$. This is relevant in the following lemma, which shows that each set $X\in \mathcal{A} \cup \mathcal{H}$ generates the unincluded $\Om1$ moves. We remark that the realizations of moves $\Om1c$ and $\Om1d$ come directly from~\cite{Pol}.

\begin{lemma}\label{Type1}
Each oriented $\Om1$ move may be realized as follows:
\begin{enumerate}
\item $\Om1a$ may be realized by a sequence of $\Om1d$ and $\Om 2d$ moves. 
\item $\Om 1b$ may be realized by a sequence of $\Om1c$ and $\Om2c$ moves.
\item $\Om 1c$ may be realized by a sequence of $\Om1b$ and $\Om2d$ moves. 
\item $\Om 1d$ may be realized by a sequence of $\Om 1a$ and $\Om2c$ moves.
\end{enumerate}
\end{lemma}

\begin{proof}
For each realization, we first introduce two crossings on a single strand with a type 2 move and then we undo the outermost twist with a type 1 move, as shown below.
\begin{align*}
\includegraphics[height=0.8in]{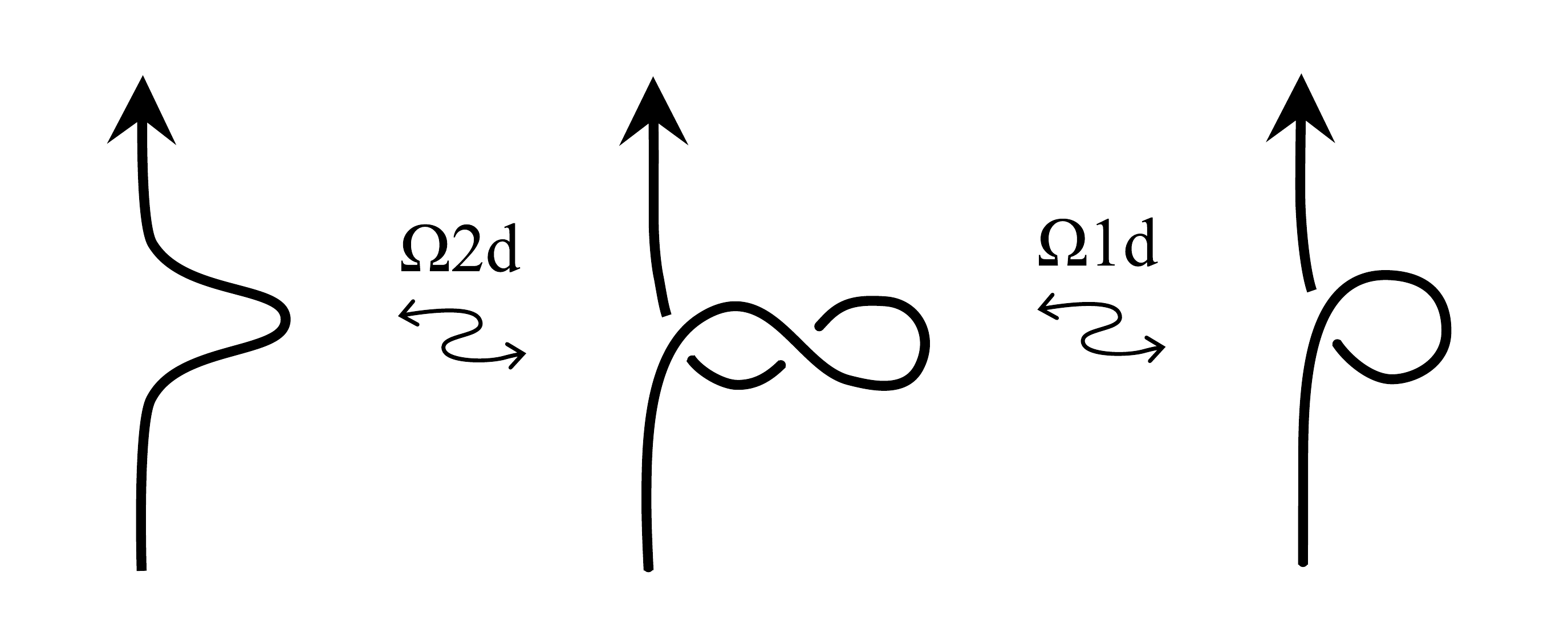} \quad& \quad\includegraphics[height=0.8in]{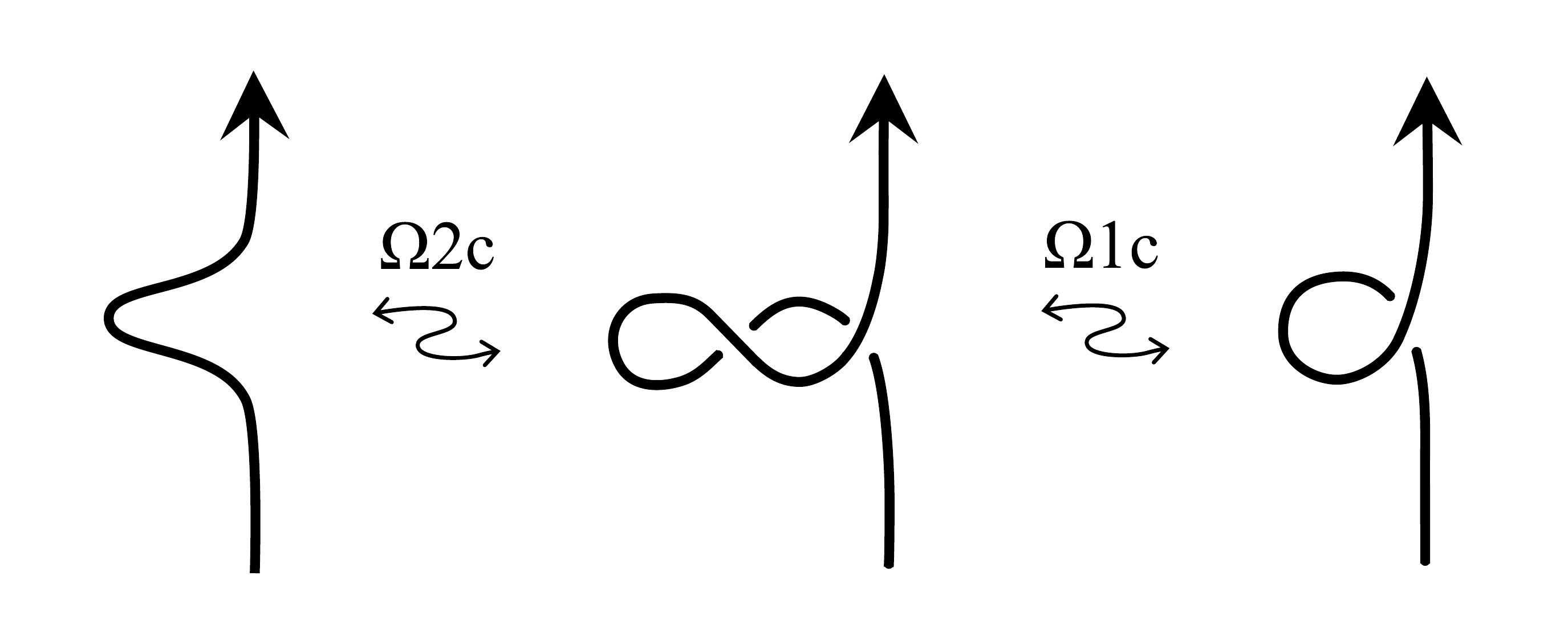}\\
\includegraphics[height=0.78in]{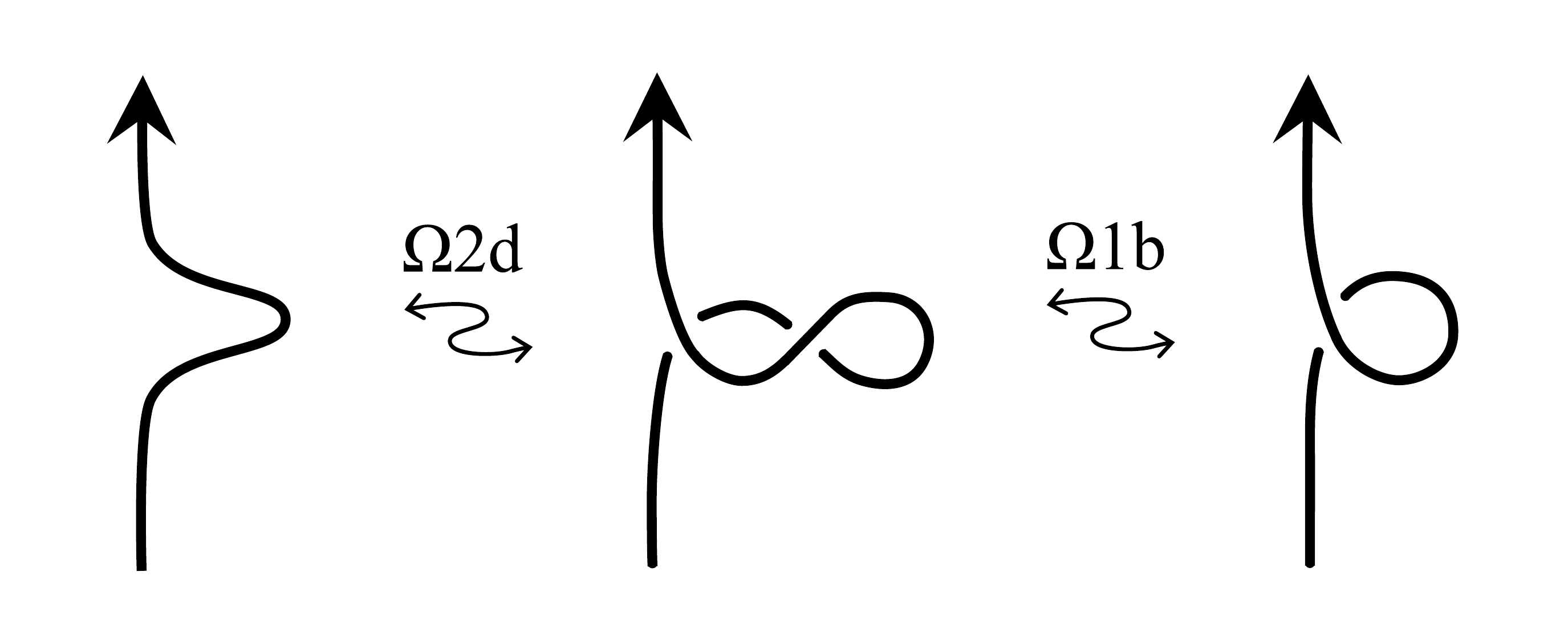}\quad & \quad\includegraphics[height=0.8in]{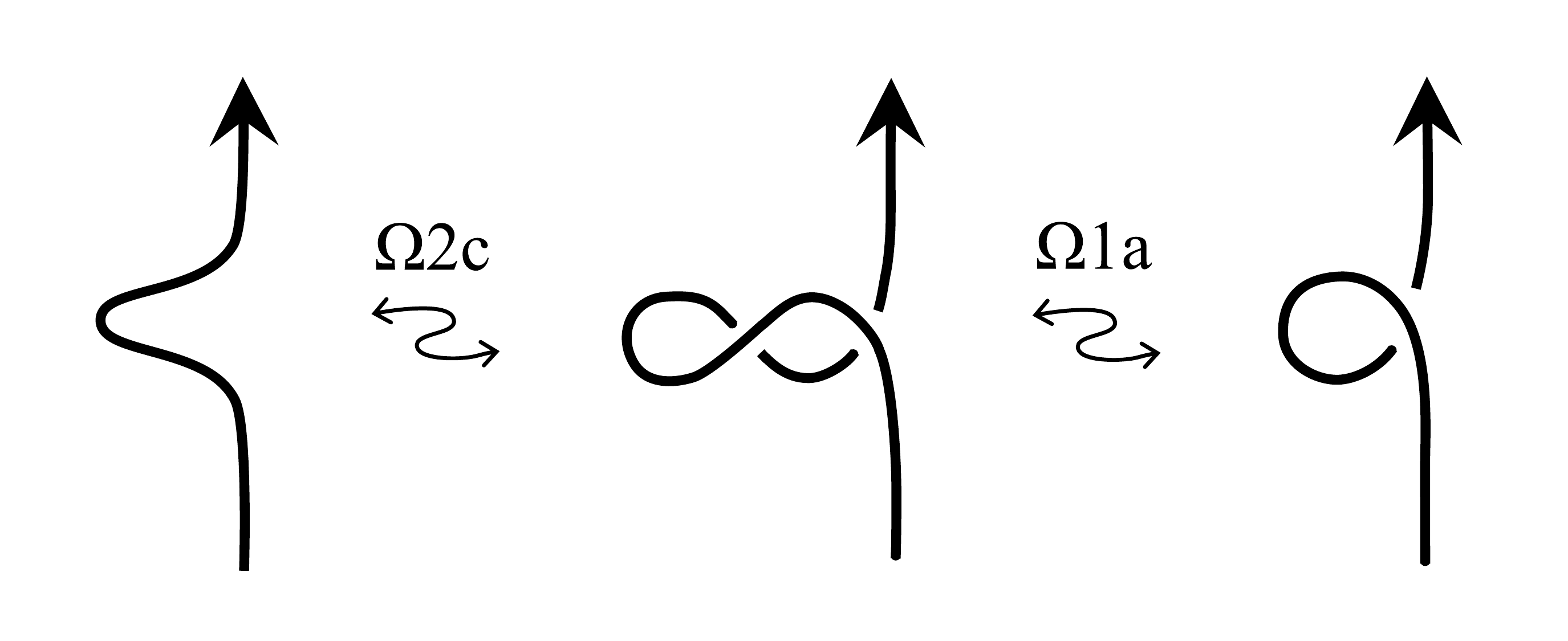}
\end{align*}
\end{proof}

Lemma~\ref{Type1} and Lemma~\ref{2c2d} taken together imply that every set $X \in \mathcal{A} \cup \mathcal{H}$ generates the unincluded oriented $\Om1$ moves and both moves $\Om2c$ and $\Om2d$.
To generate the remaining oriented $\Om2$ moves with each set $X$ in $\mathcal{A} \cup \mathcal{H}$, we must use an oriented $\Om3$ move. The next lemma, taken with the previous two, shows that if $X \in \mathcal{A}$ then $X$ generates the move $\Om3c$, and if $X \in \mathcal{H}$ then $X$ generates the move $\Om3d$. Note that the first realization is a result from Polyak in \cite{Pol}.

\begin{lemma} \label{3c3d}
The move $\Om3c$ may be realized by a sequence of moves $\Om3a, \Om2c,$ and $\Om2d$. The move $\Om3d$ may be realized by a sequence of moves $\Om3h, \Om2c,$ and $\Om2d$.
\end{lemma}
\begin{proof}
To realize the move $\Om3c$, the move $\Om2c$ is used first to pass one strand over another strand, as in the diagram below. Then we apply the move $\Om3a$ to pass that same strand over the crossing in the middle of the diagram. Finally, the realization of the move $\Om3c$ is finalized through the application of the move $\Om2d$. We remark that the moves $\Om2c$ and $\Om2d$ are applied to the same pair of strands.
 \[ \includegraphics[height=0.75in]{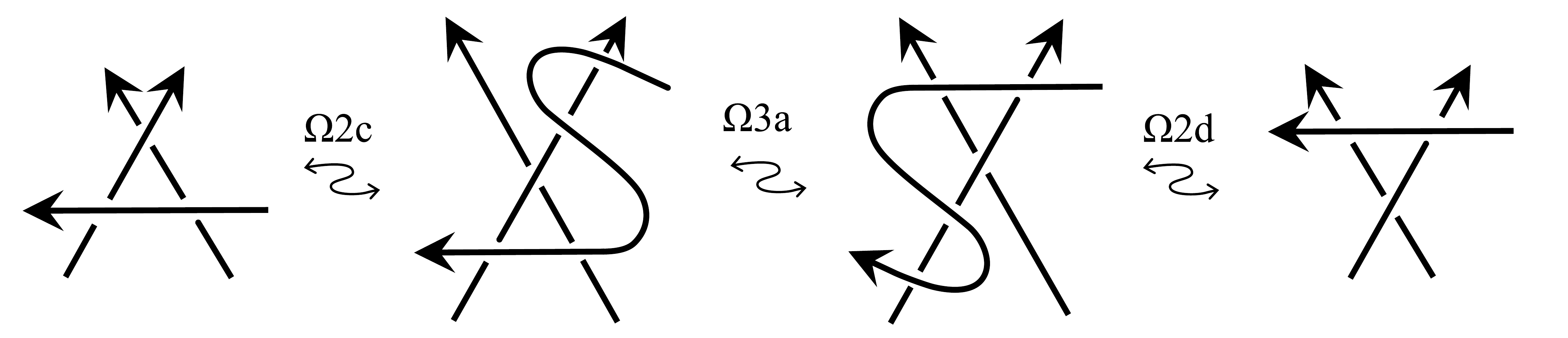}\]
 The move $\Om3d$ is realized in a similar manner, with the difference that the move  $\Om3a$ in the sequence is replaced with the move  $\Om3h$, as shown below.
  \[ \includegraphics[height=0.75in]{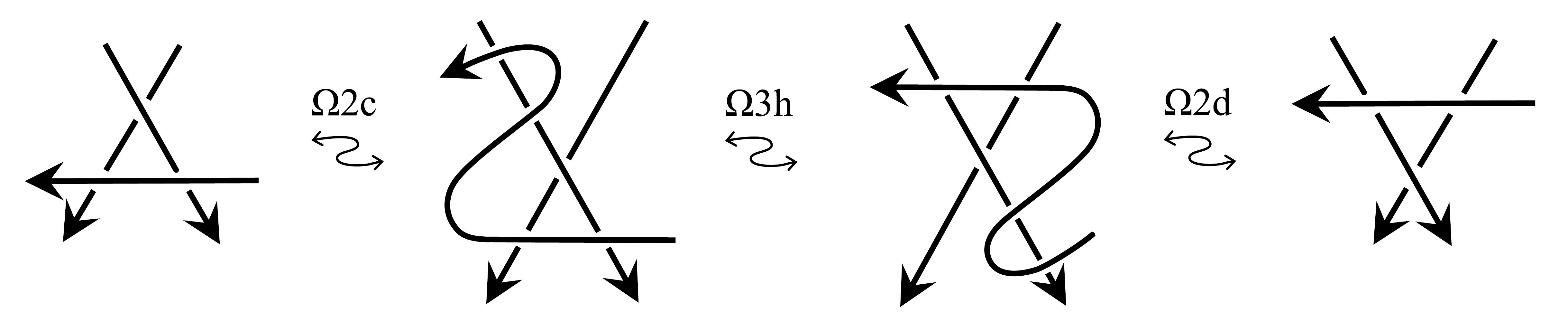}\]
Hence, the two realizations are complete and the moves $\Om3c$ and $\Om3d$ may be replaced by the respective sequences given here.
\end{proof}

Lemma~\ref{3c3d} combined with Lemmas~\ref{2c2d} and~\ref{Type1} imply that each set $X \in \mathcal{A}$ generates the move $\Om3c$ and that each set $X \in \mathcal{H}$ generates the move $\Om3d$. We use this result to prove in the following lemma that every set $X \in \mathcal{A}\cup \mathcal{H}$ generates the move $\Om2a$ or $\Om2b$, whichever is not included in the corresponding set $X$. 

\begin{lemma}\label{2a2b}
The move $\Om2a$ may be realized by each of the following sequences of moves:
\begin{multicols}{2}
\begin{itemize}
\item $\Om 1a, \Om2d, \Om3d$
\item $\Om 1c, \Om2d, \Om3c$
\end{itemize}
\end{multicols}
The move $\Om2b$ may be realized by each of the following sequences of moves:
\begin{multicols}{2}
\begin{itemize}
\item $\Om 1b, \Om2c, \Om3d$
\item $\Om 1d, \Om2c, \Om3c$
\end{itemize}
\end{multicols}
\end{lemma}

\begin{proof} We first show how to realize the move $\Om2a$.
 \begin{center}
  \includegraphics[height=0.8in]{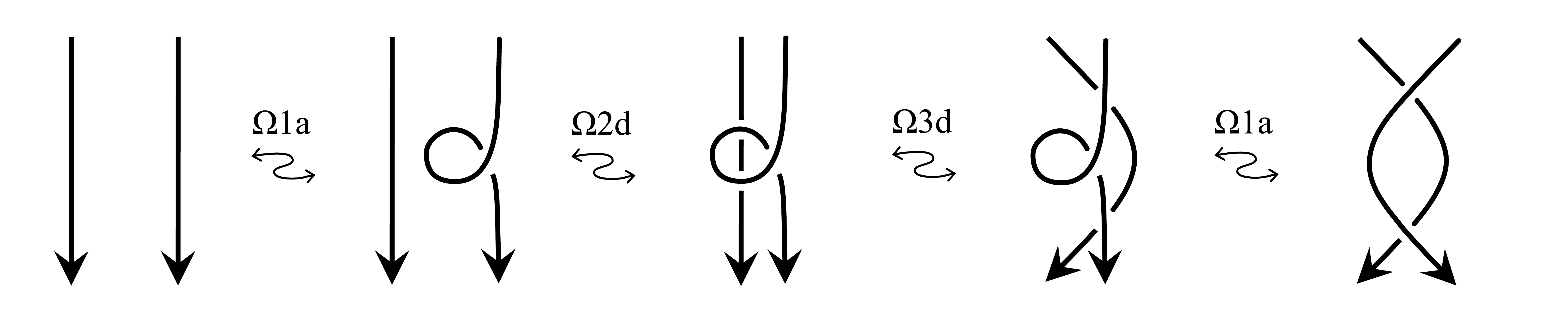}\\
 \includegraphics[height=0.8in]{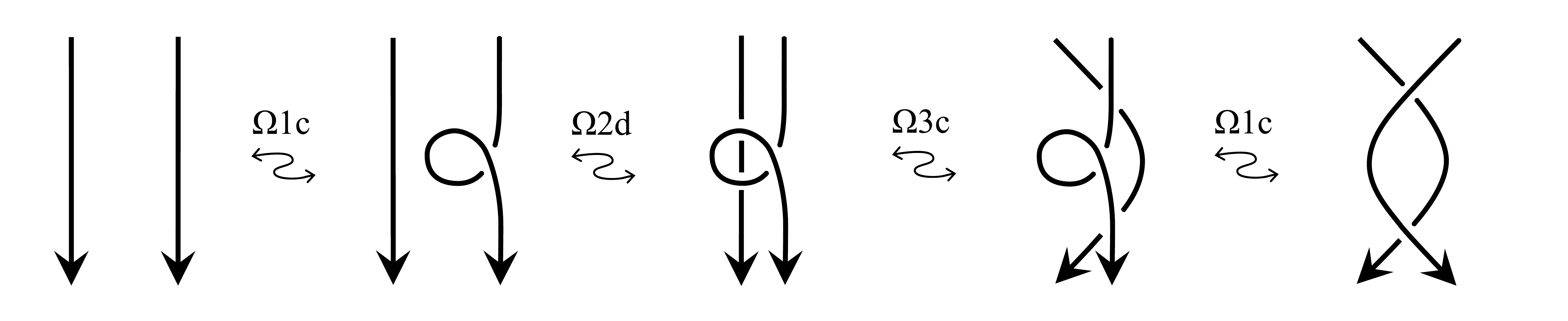}
 \end{center}
Next we realize the move $\Om2b$ in the two ways shown below; the second sequence was given by Polyak in~\cite{Pol}.
The process for these realizations is nearly identical to the process for the realizations of the move $\Om2a$.
 \begin{center}
  \includegraphics[height=0.8in]{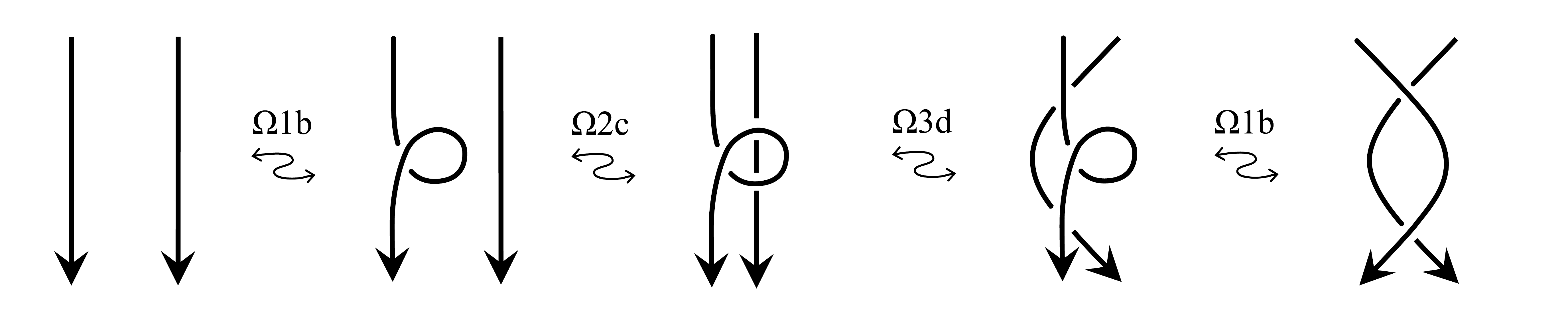}\\
  \includegraphics[height=0.8in]{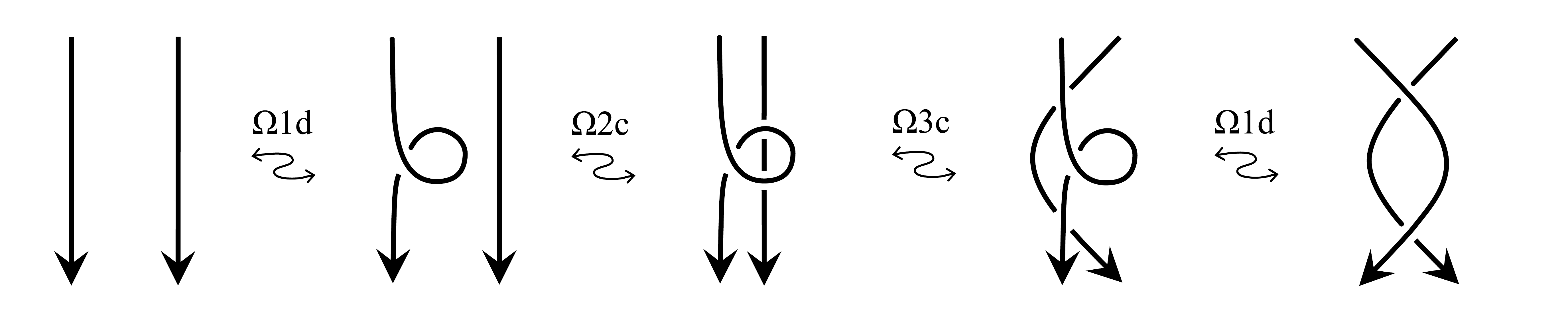}
  \end{center}
Thus we have two realizations for each of the moves $\Om2a$ and $\Om2b$.
\end{proof}

\begin{remark} \label{remark4}
Lemma~\ref{2a2b}, taken together with the previous lemmas, shows that every set $X \in \mathcal{A} \cup \mathcal{H}$ generates all oriented $\Om2$ moves. 
\end{remark}

The next step is to prove that every set $X \in \mathcal{A} \cup \mathcal{H}$ generates all oriented Reidemeister moves of type 3. We prove this in the following two lemmas.

\begin{lemma}\label{H3moves}
 Every set $X \in \mathcal{H}$ generates all oriented $\Om3$ moves. 
\end{lemma}

\begin{proof}
Let $X$ be any set element in the collection $\mathcal{H}$. Then, we know that the move $\Om3h$ is in $X$. The set $X$ also contains either the move $\Om2a$ or $\Om2b$, and from the previous lemmas, we know that $X$ generates all of the other Reidemeister moves of type 2. From Lemma~\ref{3c3d}, the move $\Om3d$ may be realized by a sequence of moves $\Om3h, \Om2c,$ and $\Om2d$.  Hence, $X$ generates the move $\Om3d$.

The move $\Om3g$ may be realized by a sequence of the moves $\Om3h, \Om2c$ and $\Om2d$, as shown below, and thus $X$ generates the move $\Om3g$.
 \[ \includegraphics[scale=0.21]{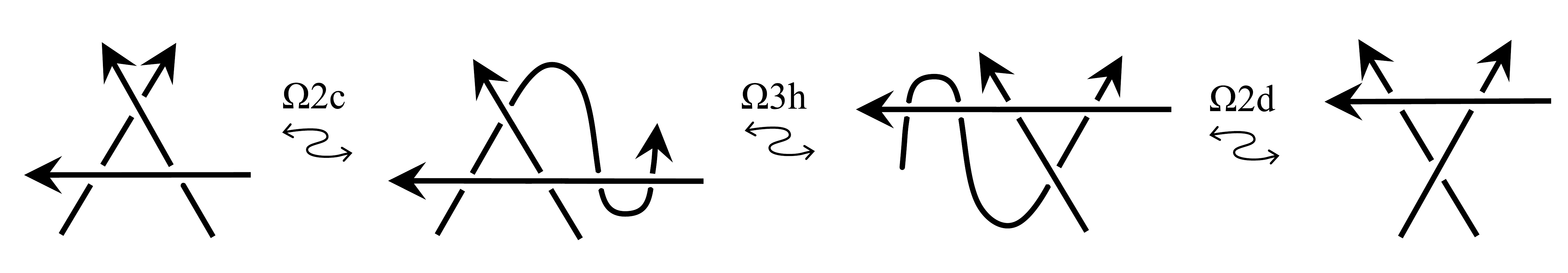}\]

Next, we show that the move $\Om3f$ may be realized by a sequence of $\Om3d, \Om2a$ and $\Om2b$ moves.
 \[ \includegraphics[scale=0.2]{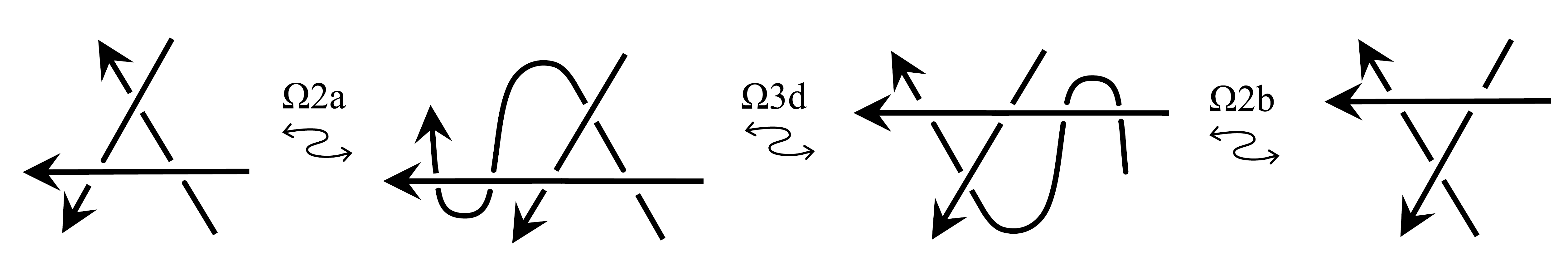}\]
Therefore, the set $X$ generates the move $\Om3f$. Now that we have been able to generate  $\Om3f$, we show that $X$ generates also the move $\Om3a$. Indeed, the move $\Om3a$ may be realized by a sequence of the moves $\Om3f, \Om2c,$ and $\Om2d$, as shown below.  \[ \includegraphics[height=0.85in]{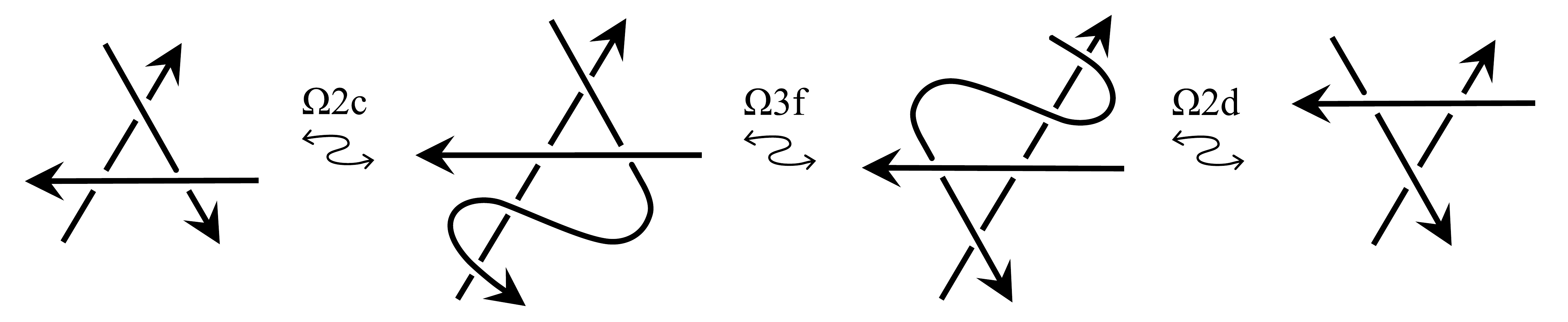}\]

From Lemma~\ref{3c3d}, we know that the move $\Om3c$ may be realized by a sequence of moves $\Om3a, \Om2c,$ and $\Om2d$. Hence, the set $X$ generates the Reidemeister move $\Om3c$. As we prove next, the set $X$ also generates the move $\Om3b$, which can be realized through a sequence of the moves $\Om3a, \Om2c$ and $\Om2d$.
 \[ \includegraphics[height=0.8in]{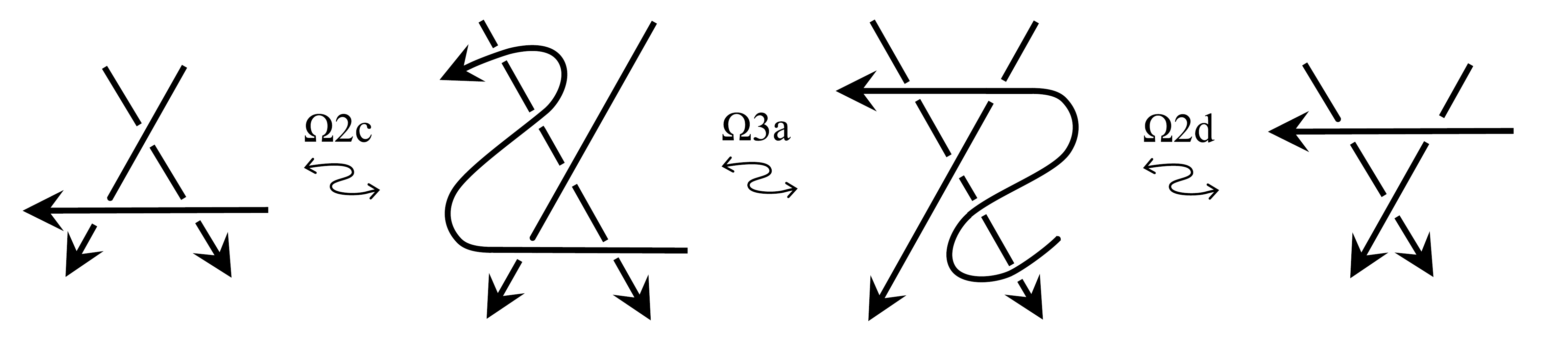}\]
 Now that we have realized the move $\Om3b$, we show that the set $X$ generates the move $\Om3e$, which can be realized by a sequence of the moves $\Om3b,\Om2a$ and $\Om2b$.
 \[ \includegraphics[height=0.8in]{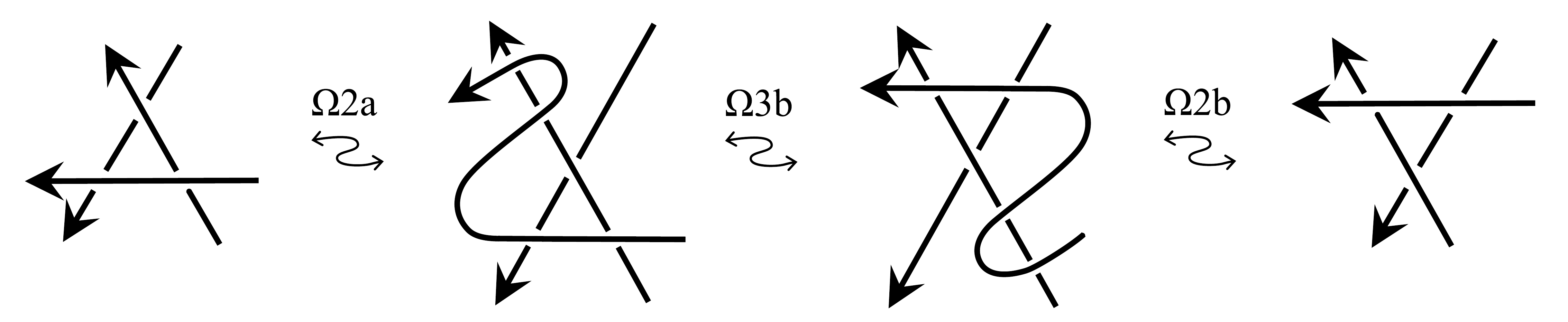}\]
 This completes the proof that $X$ generates all oriented Reidemeister moves of type 3. 
\end{proof}

\begin{lemma}\label{A3moves}
 Every set $X \in \mathcal{A}$ generates all oriented Reidemeister moves of type 3. 
\end{lemma}
\begin{proof}
Let $X$ be any set-element in the collection $\mathcal{A}$. Then $\Om3a \in X$. Moreover, the set $X$ also contains either the move $\Om2a$ or $\Om2b$ and it generates all of the other Reidemeister moves of type 2 (as shown in Lemmas~\ref{2c2d} through~\ref{2a2b}).

By Lemma~\ref{3c3d}, the move $\Om3c$ may be realized by a sequence of moves $\Om3a, \Om2c$ and $\Om2d$, and thus the set $X$ generates the move $\Om3c$. As we have  shown in Lemma~\ref{H3moves}, the move $\Om3b$ may be realized through a sequence of the moves $\Om3a, \Om2c$ and $\Om2d$; and the move $\Om3e$ can be realized by a sequence of the moves $\Om3b,\Om2a$ and $\Om2b$. Therefore, the set $X$ generates the moves $\Om3b$ and $\Om3e$.

We next prove that the move $\Om3h$ may be realized by a sequence of the moves $\Om3e, \Om2c,$ and $\Om2d$, as shown below.
 \[ \includegraphics[scale=0.21]{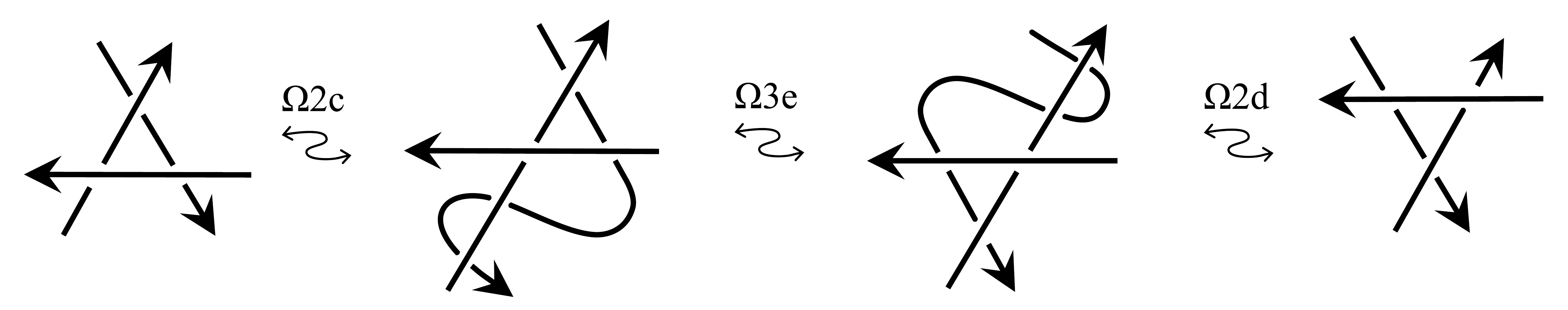}\]
Therefore, the set $X$ generates the move $\Om3h$. Since the move $\Om3h$ belongs to every set element in the collection $\mathcal{H}$, then by Lemma~\ref{H3moves}, the set $X$ generates the remaining Reidemeister moves of type 3.
Hence, every set $X \in \mathcal{A}$ generates all oriented $\Om3$ moves. 
\end{proof}

We are now ready to prove Theorem~\ref{P1Main}. The proof will demonstrate how the lemmas in this section all work together to show that every $X \in \mathcal{A} \cup \mathcal{H}$ generates all oriented Reidemeister moves. 

\begin{proof}[Proof of Theorem 2]
Lemmas~\ref{2c2d} and~\ref{Type1} combined, prove that every set $X \in \mathcal{A} \cup \mathcal{H}$ generates the moves $\Om2c$ and $\Om2d$, and the two oriented $\Om1$ moves not included in $X$.

Using the moves $\Om2c$ and $\Om2d$, Lemma \ref{3c3d} shows that each set $X\in \mathcal{A}$ generates the move $\Om3c$ and each $X \in \mathcal{H}$ generates the move $\Om3d$. With the move $\Om3c$ or $\Om3d$ at hand, Lemma \ref{2a2b} shows that $X \in \mathcal{A} \cup \mathcal{H}$ generates either $\Om2a$ or $\Om2b$ (whichever is missing from $X$). Lemmas~\ref{H3moves} and~\ref{A3moves} complete the proof of the theorem, by showing that every set $X \in \mathcal{A} \cup \mathcal{H}$ generates all of the remaining oriented $\Om3$ moves.

Therefore, any $X \in \mathcal{A} \cup \mathcal{H}$ is a generating set of oriented Reidemeister moves for knot and link diagrams.
\end{proof}


\subsection{Non-Generating Reidemeister Moves} \label{sec:NonGenSets}

In this section we prove that the 12 sets described in Theorem~\ref{P1Main} represent all minimal (4-element) generating sets of oriented Reidemeister moves for knot diagrams. To prove this, we show that certain oriented Reidemeister moves (or pairs of moves) cannot exist in a generating set containing four oriented Reidemesiter moves. We also show that a minimal generating set of oriented Reidemeister moves contains two $\Om1$ moves.

 To begin, we show in Lemma~\ref{Type1Ex} that two $\Om1$ moves are required in a generating set, and we explain why certain pairs of $\Om1$ moves do not generate all oriented Reidemeister moves.

The oriented $\Om1$ moves are the only Reidemeister moves that affect the writhe and the winding number. Recall that the \textit{writhe}, $\omega(D)$, of an oriented knot diagram $D$ is obtained by subtracting the number of negative crossings from the number of positive crossings in $D$. The \textit{winding number} (also called the \textit{rotation number}) of a diagram $D$, usually denoted by $\text{rot}(D)$, is obtained by smoothing all crossings in $D$ according to the orientation and adding the rotation numbers of the resulting oriented loops, where $\text{rot} \left ( \raisebox{-.15cm}{\includegraphics[height=.2in]{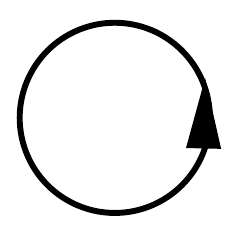}} \right ) = 1$ and $\text{rot} \left ( \raisebox{-.15cm}{\includegraphics[height=.2in]{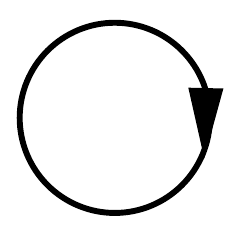}} \right ) = -1$. The writhe and winding number are regular isotopy invariants for oriented knots, and will be instrumental in showing why Theorem~\ref{P1Main} entirely excludes the pairs $(\Om1a, \Om1d)$ and $(\Om1b, \Om1c)$ from a generating set of oriented Reidemeister moves. Lemma~\ref{Type1Ex} justifies this exclusion and shows that two $\Om1$ moves are required in any generating set of oriented Reidemeister moves, and thus proving that a set with fewer than four moves cannot be a generating set. A brief proof of the next lemma was provided in~\cite{Pol}, but in order to have a self-contained paper, we provide a detailed proof here, which uses the same ideas as in~\cite{Pol}.

\begin{lemma}\cite[Lemma 3.1]{Pol} \label{Type1Ex}
Any generating set of Reidemeister moves contains at least two $\Om1$ moves. Moreover, neither of the two pairs $(\Om1a, \Om1d)$ or $(\Om1b, \Om1c)$ taken together with all $\Om2$ and $\Om3$ moves generate all of the oriented Reidemeister moves.
\end{lemma}

\begin{proof}
The move $\Om1a$ introduces (or removes, depending on the direction of the move) one positive crossing and a full clockwise rotation, which simultaneously increases (decreases) the writhe and decreases (increases) the winding number each by 1. In a similar manner, the move $\Om1d$ introduces (removes) one negative crossing and a full counter-clockwise rotation, which simultaneously decreases (increases) the writhe and increases (decreases) the winding number each by 1. Hence, moves $\Om1a$ and $\Om1d$ preserve $\omega + \text{rot}$, a quantity that is also preserved by all $\Om2$ and $\Om3$ moves. However, $\omega + \text{rot}$ is not preserved by the moves $\Om1b$ and $\Om1c$, which change $\omega + \text{rot}$ by $\pm2$. Thus, the moves $\Om1a$ and $\Om1d$ taken with all oriented Reidemeister moves $\Om2$ and $\Om3$ cannot generate moves $\Om1b$ and $\Om1c$.

In considering a set consisting of moves $\Om1b$ and $\Om1c$ taken with all oriented $\Om2$ and $\Om3$ moves, we must note that both moves $\Om1b$ and $\Om1c$ affect $\omega + \text{rot}$ by $\pm2$, while all other oriented moves preserve $\omega + \text{rot}$. Thus, any realization of an oriented Reidemeister move with this set must either not use any $\Om1b$ or $\Om1c$ moves or contains a pair of moves so that one move adds 2 to $\omega + \text{rot}$ and the other subtracts 2. As an example, if the forward $\Om1b$ move is applied, then either the inverse $\Om1b$ move or the forward $\Om1c$ move must also be applied to preserve $\omega + \text{rot}$. Similar statements apply to the forward $\Om1c$ move and the inverse moves. Each of these pairs of moves not only preserve $\omega + \text{rot}$, but they also preserve $\omega$ and $rot$ separately---two quantities that are not preserved by $\Om1a$ or $\Om1d$. Hence, the moves $\Om1b$ and $\Om1c$ taken with all oriented $\Om2$ and $\Om3$ moves will not generate moves $\Om1a$ and $\Om1d$.

For the same reasons that neither pair $(\Om1a, \Om1d)$ nor pair $(\Om1b, \Om1c)$, taken with all $\Om2$ and $\Om3$ moves, gives a generating set, we have that no single $\Om1$ move, taken with all $\Om2$ and $\Om3$ moves, is sufficient to generate all oriented Reidemeister moves. A single $\Om1$ move that preserves $\omega + \text{rot}$ (such as moves $\Om1a$ and $\Om1d$) cannot be used to generate moves that do not preserve that quantity (such as moves $\Om1b$ and $\Om1c$). Hence, two distinct $\Om1$ moves are required in any generating set of oriented Reidemeister moves.
\end{proof}

\begin{remark} \label{minimality}
It is clear that at least one oriented Reidemeister move of each type must be included in a generating set of Reidemeister moves. By Lemma~\ref{Type1Ex}, a generating set of Reidemeister moves contains at least two $\Om1$ moves. Since we have shown in Theorem~\ref{P1Main} that four moves are sufficient to generate all 16 oriented Reidemeister moves, this proves the minimality of a 4-element generating set. In particular, this proves that the 12 sets in the collections $\mathcal{A}$ and $\mathcal{H}$ are minimal generating sets of oriented Reidemeister moves.
\end{remark}

We refer to the four pairs of $\Om1$ moves, $(\Om1a, \Om1b), (\Om1a, \Om1c), (\Om1b, \Om1d)$ and $(\Om1c, \Om1d)$, in a generating set of Reidemeister moves as \textit{generating $\Om1$ moves}.

Given an oriented link or knot diagram, smooth all of its crossings according to orientation, as shown in Figure~\ref{fig:smooth}. Denote by $C^+$ and $C^-$ the number of counter-clockwise and respectively clockwise oriented circles in the resulting smoothed diagram. 
\begin{figure}[h] 
\begin{center}
\includegraphics[scale=0.2]{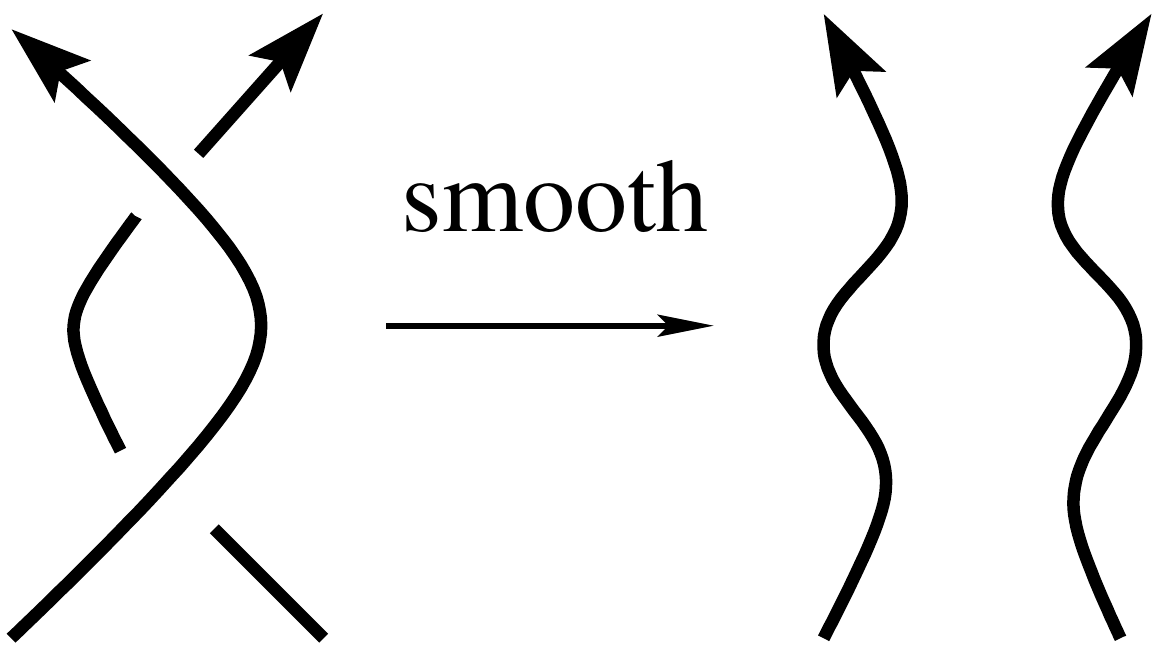} \hspace{0.75cm}  \includegraphics[scale=0.2]{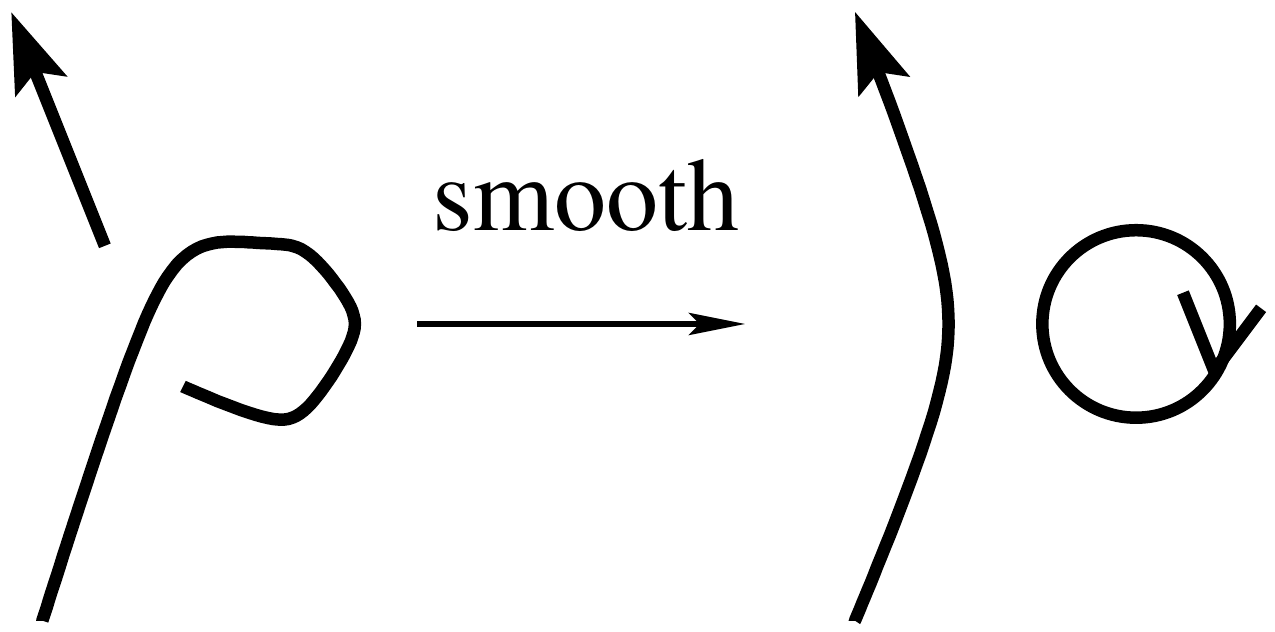}     \hspace{0.75cm}  \includegraphics[scale=0.2]{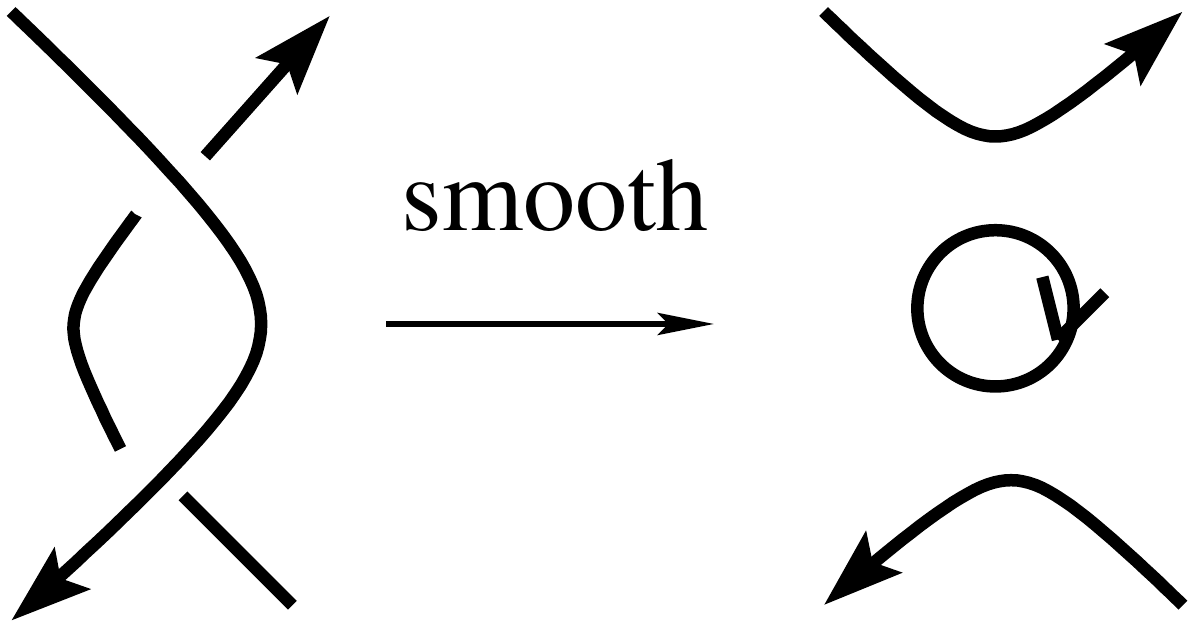}          
\end{center}
\caption{Smoothing a diagram according to orientation} 
\label{fig:smooth}
\end{figure}

\begin{remark} \label{smooth}
We collect the behavior of the oriented Reidemeister moves with respect to the smoothing of the corresponding diagrams.
\begin{enumerate}
\item [(i)]
Move $\Om1a$ changes $C^-$ and $C^+ + C^- + \omega$, while it preserves $C^+$ and $C^+ + C^- - \omega$.\\
Move $\Om1b$ changes $C^+$ and $C^+ + C^- + \omega$, while it preserves $C^-$ and $C^+ + C^- - \omega$. \\
Move $\Om1c$ changes $C^-$ and $C^+ + C^- - \omega$, while it preserves $C^+$ and $C^+ + C^- + \omega$. \\
Move $\Om1d$ changes $C^+$ and $C^+ + C^- - \omega$, while it preserves $C^-$ and $C^+ + C^- + \omega$.

\item [(ii)] 
Moves $\Om2a$ and $\Om2b$ preserve both $C^+$ and $C^-$, and thus they also preserve $C^+ + C^- \pm \omega$ since they preserve the writhe.\\
Moves $\Om2c$ and $\Om2d$ may change $C^+, C^-$, and $C^+ + C^- \pm \omega$.

\item [(iii)] 
Moves $\Om3a$ and $\Om3h$ may change $C^+, C^-$, and $C^+ + C^- \pm \omega$.\\
Moves $\Om3b$ through $\Om3g$ preserve $C^+$ and $C^-$, and thus they also preserve $C^+ + C^- \pm \omega$ since they preserve the writhe.
\end{enumerate}
\end{remark}

We use Remark~\ref{smooth} to prove that there are no minimal (4-element) generating sets of Reidemeister moves that contain one of the moves $\Om2a$ or $\Om2b$ and one $\Om3$ move distinct from the moves $\Om3a$ and $\Om3h$. In fact, we prove a stronger result.

\begin{theorem} \label{NoGenOmega3}
Let $S$ be a set formed by two generating $\Om1$ moves, both $\Om2a$ and $\Om2b$ moves, and one $\Om3$ move distinct from the moves $\Om3a$ and $\Om3h$. Then $S$ is not a generating set of Reidemeister moves.
\end{theorem}
\begin{proof}
The proof relies heavily on Remark~\ref{smooth} and it is similar to that given by Polyak in~\cite[Lemma 3.8]{Pol}, where he proves that a set formed by two generating $\Om1$ moves and the moves $\Om2a, \Om2b$ and $\Om3b$ is not generating. The main idea is that the given set $S$ does not generate the moves $\Om2c$ and $\Om2d$, and neither the moves $\Om3a$ and $\Om3h$. 

Note that the moves $\Om2a$ and $\Om2b$ preserve $C^+, C^-$ and $C^+ + C^- \pm \omega$. In addition, the pair $(\Om1a, \Om1b)$ of generating $\Om1$ moves preserves $C^+ + C^- - \omega$; the pair $(\Om1a, \Om1c)$ preserves $C^+$; the pair $(\Om1b, \Om1d)$ preserves $C^-$; and the pair $(\Om1c, \Om1d)$ preserves $C^+ +C^- + \omega$. Thus, if $S$ contains $(\Om1a, \Om1b)$, then all moves in $S$ preserve $C^+ + C^- - \omega$. If S contains the pair $(\Om1a, \Om1c)$, then all moves in $S$ preserve $C^+$. If $S$ contains $(\Om1b, \Om1d)$, then all moves in $S$ preserve $C^-$. Finally, if $S$ contains the pair $(\Om1c, \Om1d)$, then all moves in $S$ preserve $C^+ +C^- + \omega$. In each of these cases, $S$ cannot generate the moves $\Om2c, \Om2d$ and  $\Om3a, \Om3h$, since these moves may change $C^+, C^-$ and $C^+ + C^- \pm \omega$.
\end{proof}

\begin{corollary}
A minimal generating set of Reidemeister moves that contains one of the moves $\Om2a$ or $\Om2b$ must contain either the move $\Om3a$ or the move $\Om3h$. Moreover, a generating set of Reidemeister moves that contains only one move $\Om3$ that is neither the move $\Om3a$ nor the move $\Om3h$ must contain at least one of the moves $\Om2c$ or $\Om2d$.
\end{corollary}
\begin{proof} The proof follows directly from Theorem~\ref{NoGenOmega3} and the fact that a minimal, four-move, generating set of Reidemeister moves contains exactly one $\Om2$ move, one $\Om3$ move and two $\Om1$ moves.
\end{proof}

Next, we show that neither the move $\Om2c$ nor the move $\Om2d$ can be present in a minimal (4-element) generating set for all oriented Reidemeister moves. We will achieve this through the use of a few lemmas.
We start by showing that any realization of the move $\Om2a$ or $\Om2b$ by making use of either $\Om2c$ or $\Om2d$ relies on a type 3 move distinct from the moves $\Om3a$ and $\Om3h$ (which are the only type 3 moves present in the minimal generating sets of Theorem~\ref{P1Main}).

\begin{lemma}\label{2a2b3a3h}
Let $M$ be a sequence of oriented Reidemeister moves that realizes either the move $\Om2a$ or the move $\Om2b$. Suppose further that $M$ contains only one type 2 move that is either $\Om2c$ or $\Om2d$. Then $M$ must include an oriented $\Om3$ move that is neither $\Om3a$ nor $\Om3h$. 
\end{lemma}

\begin{proof}To avoid repetition, we will consider only the realization of the move $\Om2b$, since the case for the move $\Om2a$ follows similarly.
When realizing an oriented type 2 move, a different oriented type 2 move needs to be used in the localized disk of the move to be realized.  In particular, since $M$ is assumed to contain exactly one type 2 move, which is either $\Om2c$ or $\Om2d$, the proof splits into two cases.

\textbf{Case 1:} Suppose that $M$ contains the move $\Om2c$.

The move $\Om2b$ involves two strands with the same orientation in the localized disk of the move, while $\Om2c$  involves  two strands with opposite orientation. Let $D_1$ denote the left-hand side of the move $\Om2b$ containing two parallel strands with the same orientation, and let $D_2$ denote the right-hand side of the move $\Om2b$. Starting with diagram $D_1$, in order for the move $\Om2c$ to be applied, the orientation of one of the strands must locally change, which can be done only via an oriented $\Om1$ move; this move is applied to one of the strands in $D_1$ in such a way that the twist faces the other strand in $D_1$. After the application of an $\Om1$ move, the resulting diagram is ready for the application of the move $\Om2c$ between the strand without a twist and the outer edge of the twist (see for example Lemma \ref{2a2b}). The resulting diagram contains three crossings. Since the right-hand side diagram, $D_2$, of the move $\Om2b$ contains two crossing, the sequence $M$ must use the same $\Om1$ move in opposite direction to eliminate the twist, so that to ensure that the final diagram $D_2$ contains two crossings and that the sequence $M$ preserves the writhe $\omega$, a quantity preserved by the move $\Om2b$ that is to be realized. But before the $\Om1$ move can be applied in backward direction to eliminate the twist, a type 3 move needs to be applied to slide the strand over/under the crossing in the twist, so that the loop of the twist is free of crossings. Hence, $M$ contains a type 3 move.

The move $\Om2c$ may change $C^+, C^-$ and $C^+ + C^- \pm \omega$. However, when $\Om2c$ is applied between two strands where one of them is an arc $\ell$ that is part of the  loop introduced by the application of the move $\Om1$, then $\Om2c$ preserves $C^+, C^-$ and $C^+ + C^- \pm \omega$, as seen below.
\[ \includegraphics[scale=0.2]{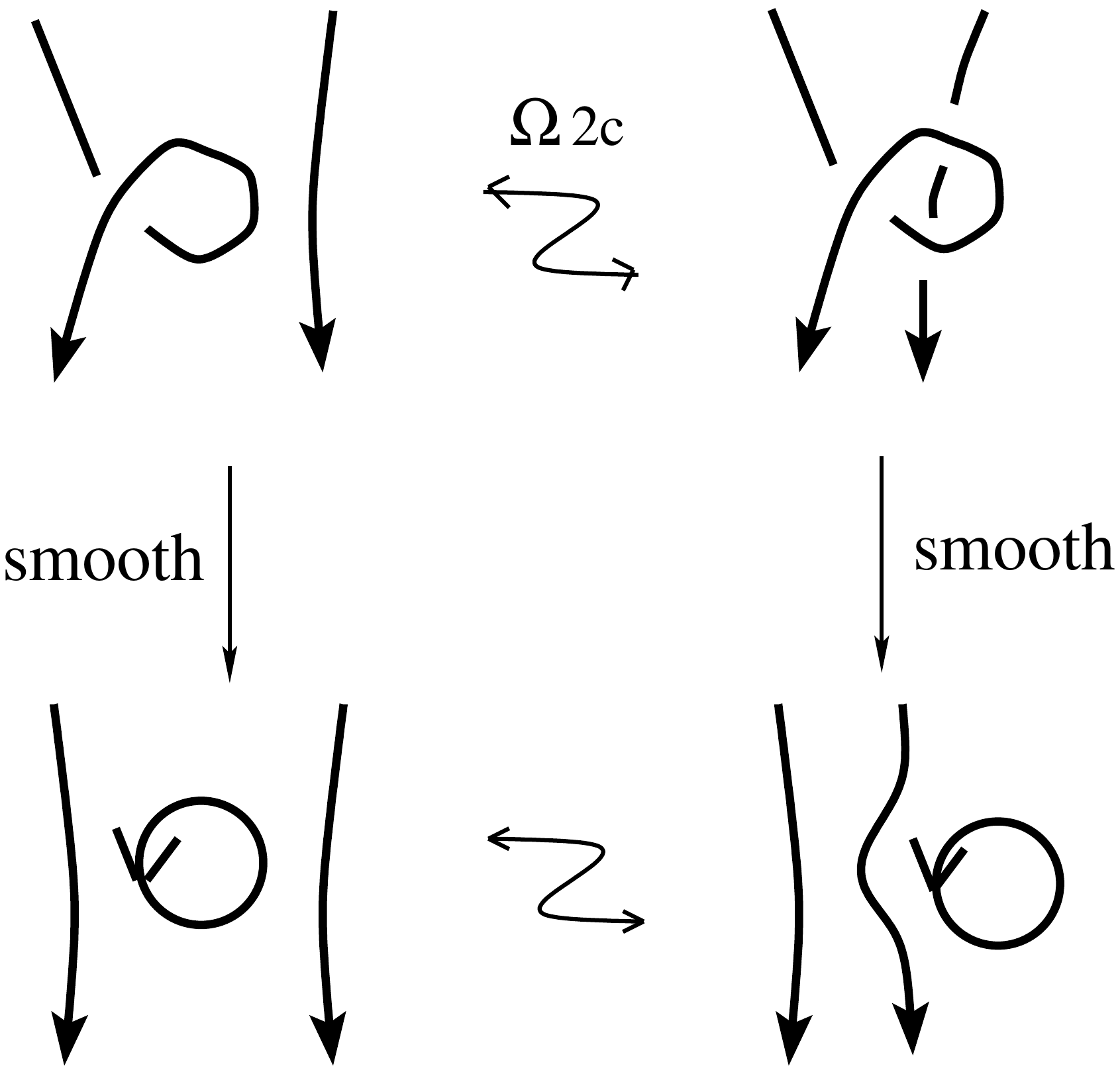}
\put(-80, 83){\fontsize{7}{7}$\ell$}
\put(-9, 75){\fontsize{7}{7}$\ell$}
 \]

So, up to this point in the realization of the move $\Om2b$, the net change for the quantities $C^+, C^-$ and $C^+ + C^- \pm \omega$ is zero. By Remark~\ref{smooth}, all of these quantities are preserved by the move $\Om2b$, therefore the type 3 move in $M$ must also preserve these quantities. Since the type 3 moves that preserve the quantities $C^+, C^-$ and $C^+ + C^- \pm \omega$ are $\Om3b$ through $\Om3g$, we conclude that $M$ includes an oriented $\Om3$ move that is neither $\Om3a$ nor $\Om3h$. This  can also be seen by noting the triangle where the type 3 move is applied is a disoriented triangle (recall that moves  $\Om3a$ and $\Om3h$ involve well-oriented triangles).

\textbf{Case 2:} Suppose that $M$ contains the move $\Om2d$.

The proof for this case is done in the same way as the previous case, since $\Om2d$ behaves similarly to the move $\Om2c$; that is, the move $\Om2d$ involves oppositely oriented strands and when applied between two strands where one of them is the outer edge of a type one move, then $\Om2d$ preserves $C^+, C^-$ and $C^+ + C^- \pm \omega$.
\end{proof} 

\begin{remark} \label{neat}
From Lemma~\ref{2a2b3a3h} we conclude that a minimal generating set that includes $\Om3a$ or $\Om3h$ as the only type 3 Reidemeister move and includes $\Om2c$ or $\Om2d$ as the only type 2 Reidemeister move must be shown to generate another type 3 move (distinct from $\Om3a$ or $\Om3h$) before it can be shown to generate either the move $\Om2a$ or the move $\Om2b$. 
\end{remark}

\begin{lemma}\label{2c2d2a2b}
Let $M$ be a sequence of oriented Reidemeister moves that realizes either the move $\Om2c$ or the move $\Om2d$. Suppose further that $M$ contains only one type 2 move that is not the move to be realized. Then $M$ must include one of the moves $\Om3a$ or $\Om3h$ in the sequence. 
\end{lemma}

\begin{proof} We will consider only the realization of the move $\Om2c$, since the move $\Om2d$ is treated similarly. When realizing the move $\Om2c$, we start with the left-hand side of the move, that is, we start from two parallel oppositely oriented strands; let $D_1$ be the corresponding diagram. In order to obtain the right-hand side diagram $D_2$ of the move $\Om2c$, a different type 2 move needs to be applied to slide a strand under/over the other.  The proof splits into three cases, depending on the type 2 move that $M$ contains.

\textbf{Case 1:} Suppose $M$ contains the move $\Om2a$.

Since the move $\Om2c$ involves oppositely oriented strands inside the localized disk, while the move $\Om2a$ involves similarly oriented strands, in order to be able to perform the move $\Om2a$, we need to change locally the orientation of one of the strands in a small neighborhood of diagram $D_1$. This can be accomplished only via an oriented type 1 move; specifically, the move $\Om1a$ or $\Om1c$ can be used on one of the strands so that the resulting twist faces the other strand of the diagram $D_1$. After the application of an $\Om1$ move, we apply the move $\Om2a$ between the strand without a twist and the outer edge of the twist, just as it was done in the realizations of move $\Om2c$ in Lemma~\ref{2c2d}.  The resulting diagram contains three crossing. Since the right-hand side diagram, $D_2$, of the move $\Om2c$ contains two crossing, the sequence $M$ must use the same $\Om1$ move in opposite direction to eliminate the twist, so that to ensure that the final diagram $D_2$ contains two crossings and that the sequence $M$ preserves the writhe $\omega$, a quantity preserved by the move $\Om2c$ that is to be realized. Since the move $\Om2a$ and a pair of forward and backward $\Om1$ move preserve $C^+$, as well as $C^-$ and $C^+ + C^- \pm \omega$, while the move $\Om2c$ may change these quantities, it follows that the sequence $M$ must contain a type 3 move that may also alter these quantities. Hence, the sequence $M$ must include one of the moves $\Om3a$ or $\Om3h$.

\textbf{Case 2:} Suppose $M$ contains the move $\Om2b$. 

This case is proved in the same way as the case for $\Om2a$, since the move $\Om2b$ involves similarly oriented strands, as $\Om2a$, and preserves $C^+, C^-$ and $C^+ + C^- \pm \omega$. 

\textbf{Case 3:} Suppose $M$ contains the move $\Om2d$. 

 Note that the moves $\Om2c$ and $\Om2d$ differ by reversal of orientation of both strands. Hence, starting with the left-hand side diagram $D_1$ of the move $\Om2c$, we need to change the orientation of both strands in a small neighborhood of diagram $D_1$; this can be accomplished only by applying an oriented $\Om1$ move to each of the strands, and the moves are performed in such a way that the twists face each other. For the diagram $D_1$, this can be accomplished by the application of the move $\Om1a$ or the move $\Om1c$. After doing so, the move $\Om2d$ can be applied to cross the outer edges of the two twists.  An example for such realization of move $\Om2c$ is shown below.
\[
 \includegraphics[scale=0.17]{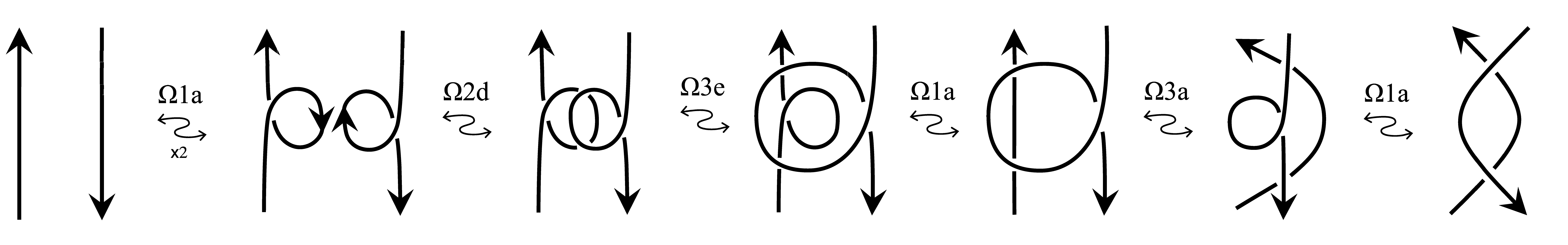} 
 \]
 Although the move $\Om2d$ may change the quantities $C^+, C^-$ and $C^+ + C^- \pm \omega$, when the move is applied to the two twists of the previously performed $\Om1$ moves, the move $\Om2d$ preserves all of these quantities, as seen below.
 \[
 \includegraphics[scale=0.25]{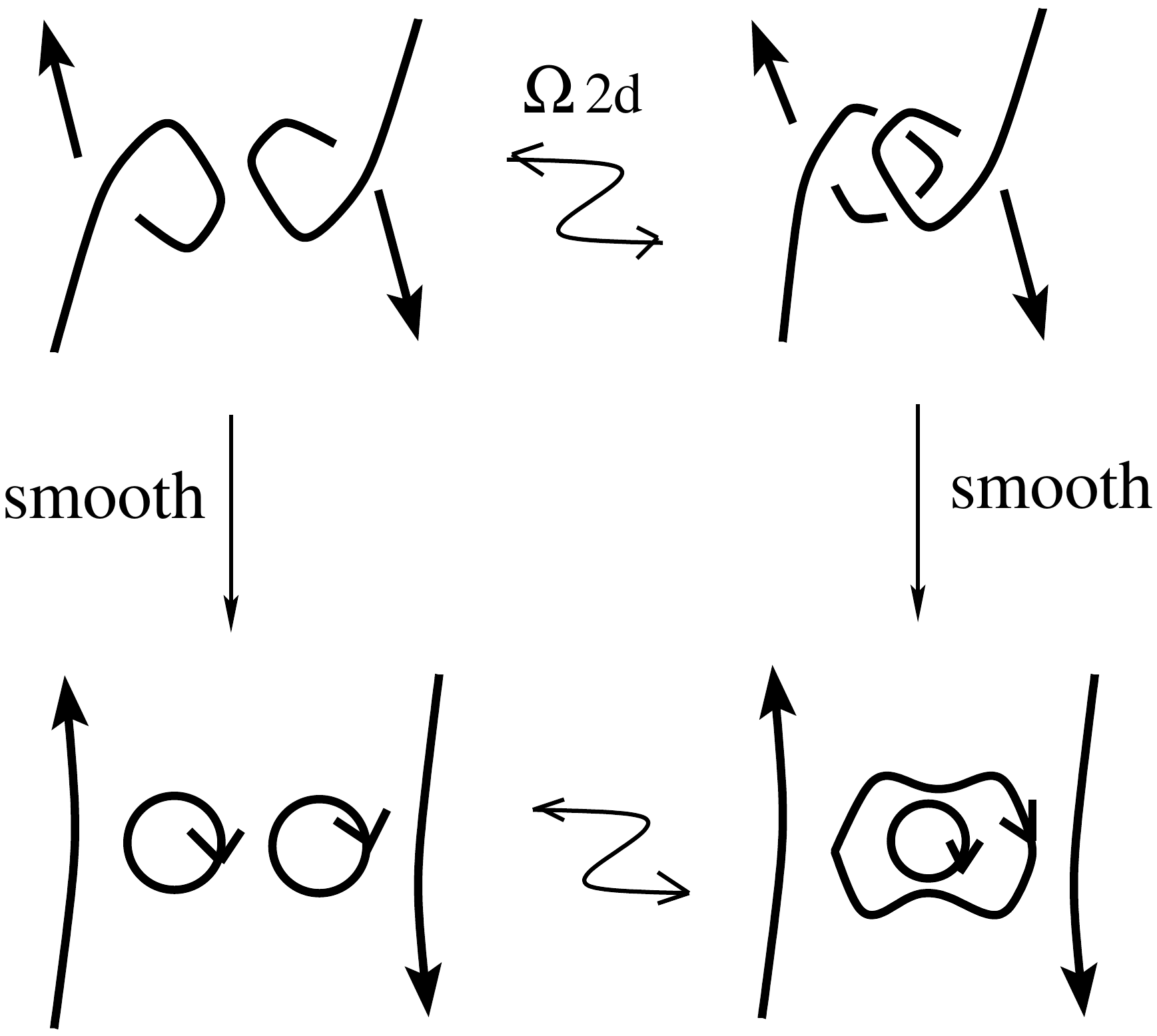} 
 \]

 As in the previous two cases, in order for the sequence $M$ to preserve the writhe $\omega$ and the final diagram to contain two crossings, the sequence $M$ must include the same $\Om1$ moves that have been used, but now the backward direction must be applied. The net change on $C^+, C^-$ and $C^+ + C^- \pm \omega$ of all of the moves present so far in $M$ is zero. Since the move $\Om2c$ that needs to be realized may change these quantities, the sequence $M$ must contain a type 3 move that changes these quantities. Hence, the sequence $M$ must include one of the moves $\Om3a$ or $\Om3h$.
 
 When realizing the move $\Om2d$, the only difference in the proof is that either $\Om1b$ move or $\Om1d$ move can be used in the first step to change the orientation of one of the strands in the left-hand side diagram of the move $\Om2d$, so that another type 2 move present in $M$ can be applied.
\end{proof}

\begin{lemma}\label{2Type3}
Let $M$ be a sequence of oriented Reidemeister moves that realizes an $\Om3$ move. Suppose further that $M$ contains exactly one type 3 move distinct from the move to be realized. Then $M$ contains two distinct $\Om2$ moves, specifically it contains either the pair $(\Om2a, \Om2b)$ or the pair $(\Om2c, \Om2d)$.
\end{lemma}

\begin{proof}
A Reidemeister move of type 3 slides a strand over, under, or through a crossing. So in order to realize an oriented $\Om3$ move through a sequence of oriented Reidemeister moves, a different $\Om3$ move must be applied to slide the strand. Consider an arbitrary $\Om3x$ move of type 3, where $x \in \{a, b, \dots, h \}$, and denote the two diagrams in $\Om3x$ by $D_1$ and $D_2$. The move $\Om3x$ slides a strand $\ell$ in $D_1$ under/over a crossing $c$, as described in Section~\ref{sec:ReidMoves} where we explained how we differentiate the eight $\Om3$ moves.  Note that both diagrams contain three crossings, where the crossings in both diagrams are either all positive, or all negative, or two are positive (or negative) and one is negative (or positive). Moreover, diagrams $D_1$ and $D_2$ divide the plane into six unbounded regions and one bounded region which is a triangle. (When referring to the six unbounded regions, we are considering the three strands in $D_1$ and $D_2$ as infinite rays, although the move $\Om3x$ is applied in a localized disk).

In order to realize the move $\Om3x$ by means of applying a different $\Om3$ move (some $\Om3$ move that is in a generating set or that has been realized already by a generating set), we need to transform the diagram $D_1$ to form another triangle that will allow the application of another $\Om3$ move to slide the strand $\ell$ under/over the crossing $c$. Any $\Om3$ move involves a triangle and in order to transform the diagram $D_1$ so that it contains another triangle, we need to increase the number of crossings in $D_1$. A type 1 move does not create a new triangle; thus a type 2 move applied to two of the strands in $D_1$ must be used, and the move needs to involve the strand $\ell$ and another strand in $D_1$. In addition, in order for the $\Om2$ move to give rise to a second triangle in the resulting diagram, the type 2 move needs to be performed in one of the unbounded regions of the diagram $D_1$ that contains three arcs and two crossings of the diagram $D_1$ (so that the move creates a triangle and not a bigon) and must be between the strand $\ell$ and one of the infinite rays emanating from the crossing $c$ (so that the strand $\ell$ can slide under/over $c$). The resulting diagram has two additional crossings, two triangles (the original triangle and a new triangle) and one bigon as bounded regions. The newly formed triangle can be now used to apply an $\Om3$ move different from $\Om3x$, to slide the strand $\ell$ under/over crossing $c$.  After the application of the move $\Om3$, the resulting diagram, call it $D'$, has five crossings and three bounded regions (two triangles and one bigon). But the final diagram of the sequence of moves in $M$ is $D_2$, which has three crossings and only one bounded region (a triangle). Thus another $\Om2$ move needs to be applied to diagram $D'$ to eliminate the extra two crossings and the bigon (along with the adjacent triangle).

The first and second type 2 moves involve adjacent bigons formed by the same pair of strands in the original diagram $D_1$ and share a crossing in $D'$, where one of the strands is $\ell$. Hence these oriented $\Om2$ moves are of different type. Moreover, if the strands involved in the type 2 moves are similarly oriented, then the two Reidemeister moves are $(\Om2a, \Om2b)$, and if the strands are oppositely generated, then these type 2 moves are $(\Om2c, \Om2d)$. Several examples of realizations of oriented $\Om3$ moves are given in Lemmas~\ref{3c3d},~\ref{H3moves} and~\ref{A3moves}.
\end{proof}

A specific case of this lemma follows, which will help to show that the only oriented $\Om3$ moves in a minimal (4-element) generating set of all oriented Reidemeister moves are $\Om3a$ and $\Om3h$.

\begin{corollary}\label{3a3h2c2d}
Let $M$ be a sequence of oriented Reidemeister moves that realizes the move $\Om3a$ or $\Om3h$, and suppose that $M$ contains exactly one type 3 move distinct from the move to be realized. Then $M$ contains the pair $(\Om2c, \Om2d)$. 
\end{corollary}

\begin{proof}
This follows as a direct application of Lemma~\ref{2Type3} and Remark~\ref{smooth}. When the type 3 move to be realized is either $\Om3a$ or $\Om3h$, the triangle in the diagram $D_1$ in the proof of Lemma~\ref{2Type3} is well-oriented. The newly created triangle in the sequence of moves $M$ from Lemma~\ref{2Type3} is adjacent to the triangle in $D_1$, as they share a side and another side of both triangles is part of the same strand in the diagram. Hence, the newly created triangle is disoriented. Thus, the $\Om3$ move corresponding to the newly created triangle is one of the moves $\Om3b$ through $\Om3g$ (a type 3 move distinct from $\Om3h$ and $\Om3a$) when the move to be realized is either $\Om3a$ or $\Om3h$.

By Remark~\ref{smooth}, the moves $\Om3b$ through $\Om3g$ preserve $C^+ + C^- \pm 1$, and the same applies to the pair $(\Om2a, \Om2b)$. Since the moves $\Om3a$ and $\Om3h$ may change $C^+ + C^- \pm 1$, we conclude that the sequence $M$ of moves that realizes $\Om3a$ or $\Om3h$ must use a pair of type 2 moves that may also change this quantity. That is, the pair of type 2 moves used in a realization of the move $\Om3a$ or $\Om3h$ is $(\Om2c, \Om2d)$. 
\end{proof}

We have the lemmas necessary to prove that the moves $\Om2c$ and $\Om2d$ do not exist in a minimal (4-element) generating set, and that either $\Om3a$ move or $\Om3h$ move must be in a minimal generating set of oriented Reidemeister moves. 

\begin{theorem} \label{3a3h}
A minimal generating set of oriented Reidemeister moves contains either the move $\Om3a$ or the move $\Om3h$.
\end{theorem}

\begin{proof}
We provide a proof by contradiction. Let $Y$ be a minimal generating set of oriented Reidemeister moves. Hence $Y$ contain four moves. Moreover, by Remark~\ref{minimality}, $Y$ contains exactly two $\Om1$ moves, one $\Om2$ move, and one $\Om3$ move. Suppose that the type 3 move that $Y$ contains is neither the move $\Om3a$ nor the move $\Om3h$.
From Corollary~\ref{3a3h2c2d} we know that the realization of the move $\Om3a$ or $\Om3h$ relies on both moves $\Om2c$ and $\Om2d$. As $Y$ contains exactly one type 2 move, $Y$ needs to generate at least one of the moves $\Om2c$ or $\Om2d$. But by Lemma~\ref{2c2d2a2b}, we know that the realization of the move  $\Om2c$ or $\Om2d$ must include move $\Om3a$ or move $\Om3h$ in the sequence. Thus, $Y$ cannot generate all of the oriented Reidemeister moves, which is a contradiction.

Therefore, any minimal generating set of oriented Reidemeister moves must contain either the move $\Om3a$ or the move $\Om3h$. 
\end{proof}

\begin{theorem} \label{2c2dEx}
A minimal generating set of oriented Reidemeister moves contains either the move $\Om2a$ or the move $\Om2b$.
\end{theorem}

\begin{proof} Let $Y$ be a minimal generating set of oriented Reidemeister moves, and suppose by contradiction that $\Om2a \not \in Y$ and $\Om2b \not \in Y$. Since $Y$ is a minimal generating set, $Y$ contains exactly one oriented $\Om2$ move, and thus we must either have $\Om2c \in Y$ or $\Om2d \in Y$. Moreover, by Theorem~\ref{3a3h}, $Y$ contains either the move $\Om3a$ or the move $\Om3h$.

 By Remark~\ref{neat}, before realizing the moves $\Om2a$ and $\Om2b$, the set $Y$ must be shown that it generates a type 3 move that is not in $Y$. Moreover, by Lemma~\ref{2c2d2a2b}, any realization of the move $\Om2d$ or $\Om2c$ (whichever is not in the set $Y$) makes use of one of the moves $\Om3a$ or $\Om3h$. Finally, by Lemma~\ref{2Type3}, the realization of an oriented $\Om3$ move requires either the pair $(\Om2a, \Om2b)$ or the pair $(\Om2c, \Om2d )$. Hence, $Y$ is not a generating set, and we have reached a contradiction. Therefore, none of the moves $\Om2c$ and  $\Om2d$ belong to a minimal generating set of oriented Reidemeister moves. 
\end{proof}

We have shown that a minimal generating set of oriented Reidemeister moves contains two $\Om1$ moves, except for the pairs $(\Om1a, \Om1d)$ and $(\Om1b, \Om1c)$, contains either $\Om2a$ or $\Om2b$, and either $\Om3a$ or $\Om3h$. These restrictions, however, suggest the existence of 16 minimal generating sets, but there are only 12 sets contained within the collection $\mathcal{A}\cup\mathcal{H}$. In the following lemma, we prove that the remaining four sets do not generate all oriented Reidemeister moves.

\begin{lemma}\label{last}
Each set of moves $\{\Om1a, \Om1b, \Om3h\}$ and $\{\Om1c, \Om1d, \Om3a\}$, taken with either move $\Om2a$ or move $\Om2b$, does not generate all oriented Reidemeister moves.
\end{lemma}

\begin{proof}
We show that the sets $\{\Om1a, \Om1b, \Om3h\}$ and $\{\Om1c, \Om1d, \Om3a\}$, taken with either $\Om2a$ or $\Om2b$, do not generate the moves $\Om2c$ and $\Om2d$. We prove this by heavily relying on Lemma~\ref{2c2d2a2b} and its proof, together with Corollary~\ref{3a3h2c2d}. By Corollary~\ref{3a3h2c2d}, in order for the given sets to realize the move $\Om3a$ or $\Om3h$ that is not present in the set, we need to have at our disposal the pair $(\Om2c, \Om2d)$. In addition, by Lemma~\ref{2c2d2a2b}, in order to generate any of the moves $\Om2c$ or $\Om2d$, one of the moves $\Om3a$ or $\Om3h$ is needed.

As seen in the proof of Lemma~\ref{2c2d2a2b}, when using either the move $\Om2a$ or move $\Om2b$ to generate the move $\Om2c$ (these are Cases 1 and 2 in the proof of  Lemma~\ref{2c2d2a2b}), we need to apply a type 1 move that introduces a right-handed twist in either of the two strands. Hence, either the move $\Om1a$ or the move $\Om1c$ must be used. The move $\Om1a$ involves a positive crossing, while the move $\Om1c$ involves a negative crossing. One of the subsequent steps in a realization of the move $\Om2c$ is to cross the strands with either $\Om2a$ or $\Om2b$ move. This introduces two new crossings, one positive and one negative; the resulting diagram contains now three crossings. Continuing with the realization of the move $\Om2c$, either move $\Om3a$ or move $\Om3h$ is needed to move the strand over/under the crossing introduced by the type 1 move, as explained in the proof of Lemma~\ref{2c2d2a2b}.  Below are examples of resulting diagrams after the application of the type 3 move for the case of an initial $\Om1a$ move (left) and, respectively, $\Om1c$ move (right).
\[
 \includegraphics[height=0.7in]{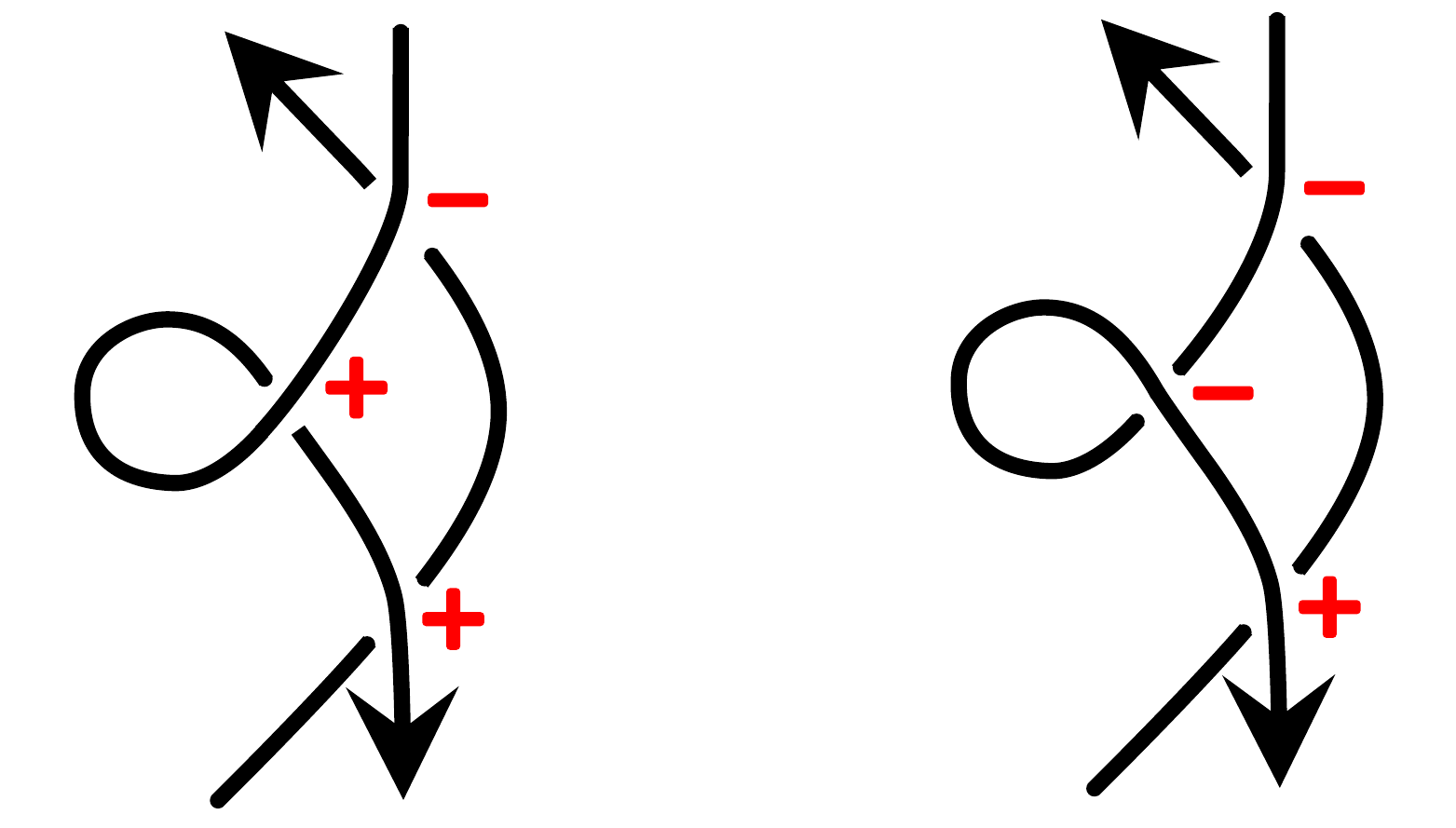} 
 \]
The move $\Om3a$ involves two positive crossings and one negative crossing, while move $\Om3h$ involves two negative crossings and one positive crossing. Therefore, in order to generate the move $\Om2c$ by using either $\Om2a$ or $\Om2b$, one of the following will apply:
\begin{itemize}
\item If the move $\Om1a$ is used, then the type 3 move used is $\Om3a$.
\item If the move $\Om1c$ is used, then the type 3 move used is $\Om3h$.
\end{itemize}

By the proof of Lemma~\ref{2c2d2a2b}, when realizing the move $\Om2d$ through the use of one of $\Om2a$ or $\Om2b$ move, we need to apply a type 1 move that introduces a left-handed twist in either of the two strands. Therefore, either the move $\Om1b$ or the move $\Om1d$ must be used. Note that the move $\Om1b$ introduces a positive crossing while the move $\Om1c$ introduces a negative crossing. Then, similar to the case above for the realization of the move $\Om2c$, we have that in order to realize the move $\Om2d$ by using either $\Om2a$ or $\Om2b$, one of the following will apply:
\begin{itemize}
\item If the move $\Om1b$ is used, then the type 3 move used is $\Om3a$.
\item If the move $\Om1d$ is used, then the type 3 move used is $\Om3h$.
\end{itemize}

We note also that the realization of an $\Om1$ move that is not at our disposal in the set requires either $\Om2c$ move or $\Om2d$ move (since the realization of a type 1 move requires a combination of at least two moves, where one of them is a type 1 move and another is a type 2 move involving oppositely oriented strands due to the fact that the type 1 move to be realized involves just one strand in a diagram).

The above discussions combined with the statement of Corollary~\ref{3a3h2c2d}, imply that the set of moves $\{\Om1a, \Om1b, \Om3h\}$ and $\{\Om1c, \Om1d, \Om3a\}$, taken with $\Om2a$ or $\Om2b$, do not generate the moves $\Om2c$ and $\Om2d$, and thus they do not generate all oriented Reidemeister moves.
\end{proof}
 
 The results of this section prove that the 12 sets of moves in the collections $\mathcal{A} \cup \mathcal{H}$ are all the minimal generating sets of oriented Reidemeister moves for oriented knot diagrams. We formally state this result with the theorem below.

\begin{theorem}\label{P1Converse}
If $X$ is a minimal (4-element) generating set of oriented Reidemeister moves for knot diagrams, then $X \in \mathcal{A}\cup\mathcal{H}$, where:
\begin{align*}
\mathcal{A} = \{&\{\Om1a, \Om1c, \Om2a, \Om3a\}, \{\Om1a, \Om1c, \Om2b, \Om3a\}, \{\Om1b, \Om1d, \Om2a, \Om3a\},\\
 &\{\Om1b, \Om1d, \Om2b, \Om3a\}, \{\Om1a, \Om1b, \Om2a, \Om3a\}, \{\Om1a, \Om1b, \Om2b, \Om3a\}\}\\
\mathcal{H} = \{&\{\Om1a, \Om1c, \Om2a, \Om3h\}, \{\Om1a, \Om1c, \Om2b, \Om3h\}, \{\Om1b, \Om1d, \Om2a, \Om3h\},\\
 &\{\Om1b, \Om1d, \Om2b, \Om3h\}, \{\Om1c, \Om1d, \Om2a, \Om3h\}, \{\Om1c, \Om1d, \Om2b, \Om3h\}\}.
\end{align*}
\end{theorem}


\section{Generating sets of moves for isotopic spatial trivalent graph diagrams} \label{sec:GeneratingSets-graphs}

The goal of this section is to extend the results in Section~\ref{generating sets} to determine minimal generating sets of oriented Reidemeister-type moves for oriented spatial trivalent graph diagrams. As before, we say that a generating set $S$ of oriented Reidemeister-type moves for oriented spatial trivalent graph diagrams is minimal, if  there is no generating set $R$ of oriented Reidemeister-type moves such that $|R| < |S|$. 

As specified in the introduction, an isotopy of unoriented spatial trivalent graphs is generated by the moves $\Om 1-\Om 5$ depicted in Figure~\ref{SpatialMoves} (see~\cite{Ca}). Here we work with oriented (or directed) spatial trivalent graphs with trivalent vertices that are either sources or sinks,  as depicted in Figure~\ref{sourcesink}. 
\begin{figure}[h]
\begin{center}
\includegraphics[scale=.2]{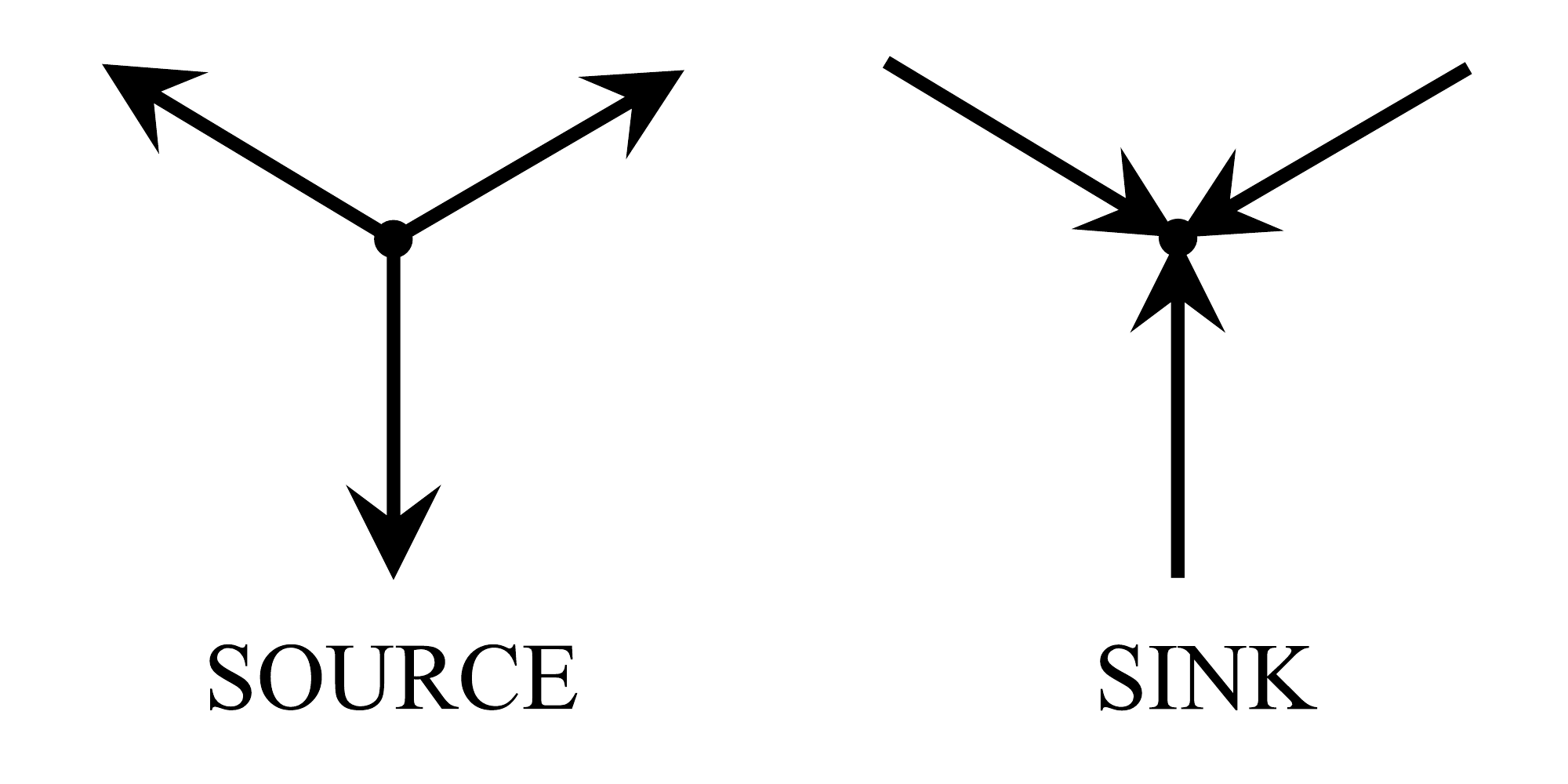}
\end{center}
\caption{Types of vertices in oriented spatial trivalent graphs}
\label{sourcesink}
\end{figure}

 For this, we need to work with oriented versions of Reidemeister-type moves for oriented spatial trivalent graph diagrams.
Figure~\ref{Type4Moves} displays the eight oriented versions of the move $\Om4$.  
\begin{figure}[h]
\begin{center}
\includegraphics[scale=.2]{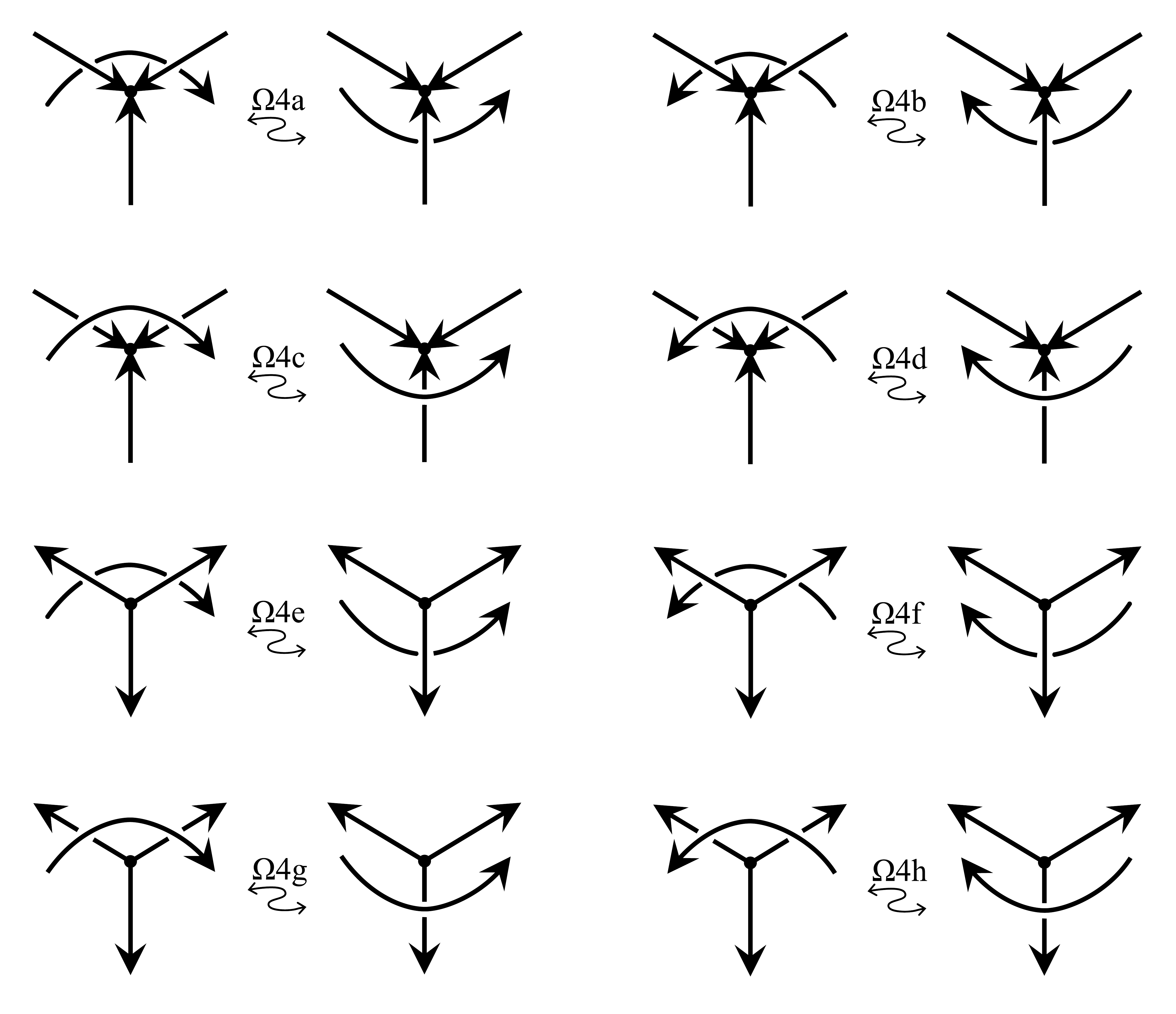}
\end{center}
\caption{Oriented Reidemeister-type moves $\Om4$}
\label{Type4Moves}
\end{figure}

Moves $\Om4a$--$\Om4d$ apply to sink vertices, while moves $\Om4e$--$\Om4h$ apply to source vertices. More specifically, during moves $\Om4a$ and $\Om4b$ ($\Om4c$ and $\Om4d$) a strand passes under (over) a sink vertex. Within the localized disk in which the move is applied, the moves $\Om4a$ and $\Om4d$ make a transition between one negative crossing and two positive crossings, but moves $\Om4b$ and $\Om4c$ make a transition between one positive crossing and two negative crossings. 
Similar distinctions are made among the moves $\Om4e$--$\Om4h$. During the moves $\Om4e$ and $\Om4f$ a strand passes under a source vertex, and during the moves $\Om4g$ and $\Om4h$ a strand passes over a source vertex. The moves $\Om4e$ and $\Om4h$ make a transition between one positive crossing and two negative crossings, while the moves $\Om4f$ and $\Om4g$ make a transition between one negative crossing and two positive crossings.

There are 4 oriented versions of the move $\Om5$, shown in Figure~\ref{Type5Moves}. The moves $\Om5a$ and $\Om5b$ involve sink vertices while moves $\Om5c$ and $\Om5d$ involve source vertices. Moreover, depending on the direction of the move, moves $\Om5a$ and $\Om5c$ introduce or remove a positive crossing; similarly, moves $\Om5b$ and $\Om5d$ introduce or remove a negative crossing.
\begin{figure}[h]
\begin{center}
\includegraphics[scale=.2]{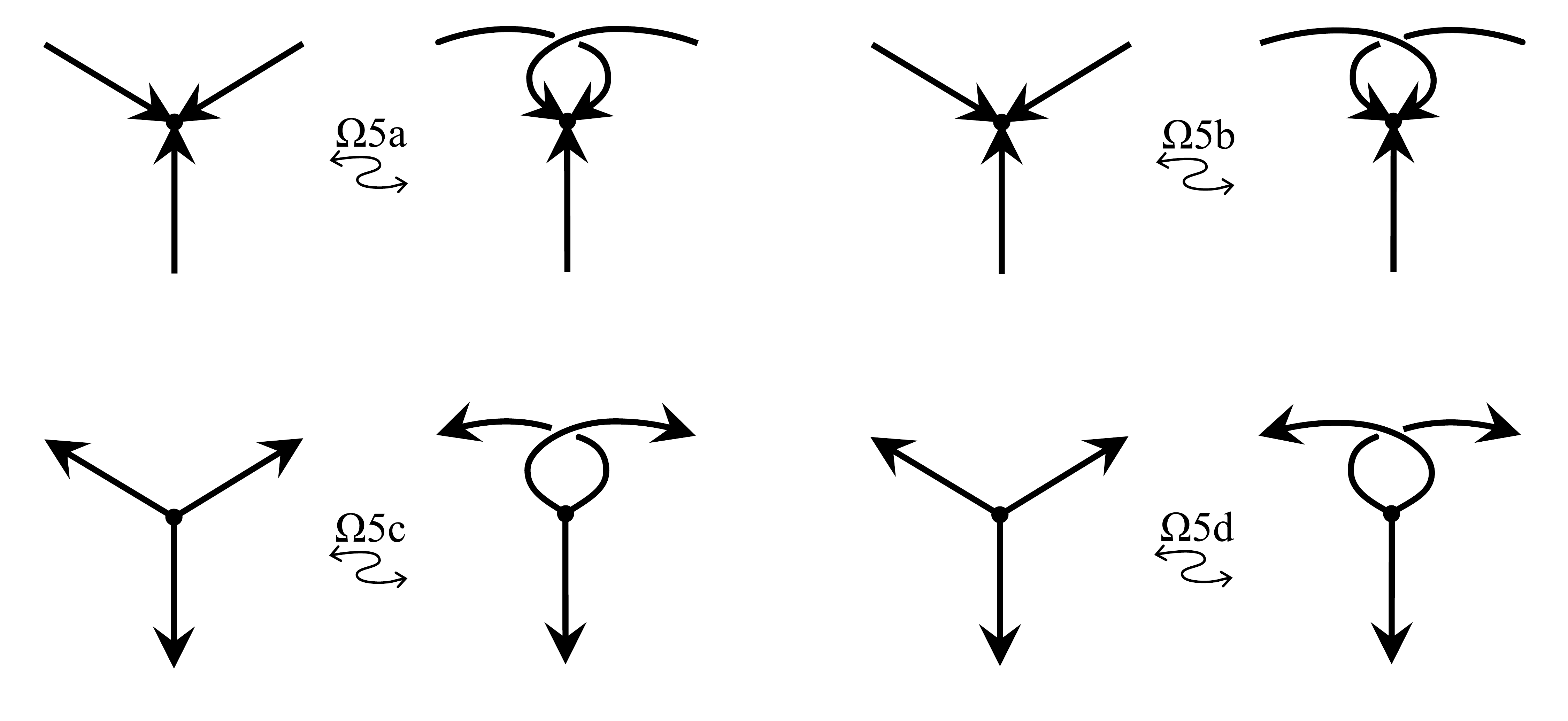}
\end{center}
\caption{Oriented Reidemeister-type moves $\Om5$}
\label{Type5Moves}
\end{figure}

In order to find all minimal generating sets of oriented Reidemeister-type moves for spatial trivalent graph diagrams, we need to determine the minimal number of oriented $\Om4$ and $\Om5$ moves required in a generating set of oriented Reidemeister-type moves. Then, we append those oriented type 4 and type 5 moves to the minimal generating sets for oriented knot diagrams listed in the collection $\mathcal{A}\cup\mathcal{H}$. Note that in the realizations of the oriented Reidemeister-type moves $\Om4$ and $\Om5$, we will apply classical Reidemeister moves $\Om1, \Om2$ and $\Om3$ (which have all been shown to be realizable from any set in the collection $\mathcal{A}\cup\mathcal{H}$).

\begin{theorem}\label{Type4Gen}
A minimal generating set of Reidemeister-type moves for oriented spatial trivalent graph diagrams with source and sink vertices contains four $\Om4$ moves, with one from each of the sets $\{\Om4a, \Om4b\}$, $\{\Om4c, \Om4d\}$, $\{\Om4e, \Om4f\}$, and $\{\Om4g, \Om4h\}$.
\end{theorem}

\begin{proof}
The oriented Reidemeister-type moves $\Om4a$--$\Om4d$ are defined on sink vertices, and moves $\Om4e$--$\Om4h$ are defined on source vertices. From this we can already conclude that a generating set of oriented Reidemeister-type moves for spatial trivalent graph diagrams must include at least two moves of type 4, specifically one move from $\Om4a$--$\Om4d$ and one move from $\Om4e$--$\Om4h$. Moreover, it is known (see for example~\cite[Section 2]{Ca}) that for the unoriented case, both versions of the move $\Om4$ from Figure~\ref{SpatialMoves} are needed, one that slides a strand under the vertex and another that slides a strand over the vertex. Hence, a generating set of oriented Reidemeister-type moves for spatial trivalent graph diagrams must contain at least two moves from $\Om4a$--$\Om4d$ and at least two moves from $\Om4e$--$\Om4h$; that is, for each type of vertex, there needs to be at least one $\Om4$ move that slides a strand under the vertex and another $\Om4$ move that slides a strand over the vertex. We show that exactly two such moves are needed for each type of vertex. 

The moves $\Om4a$ and $\Om4b$ allow passing a strand under a sink vertex. In the first row below, we show that the move $\Om4a$ can be realized through the use of the move $\Om4b$ together with two oriented $\Om2$ moves. In the second row below, we show that the move $\Om4b$ can be obtained by applying the move $\Om4a$ together with two oriented $\Om2$ moves.
\begin{center}
\noindent\includegraphics[scale=0.2]{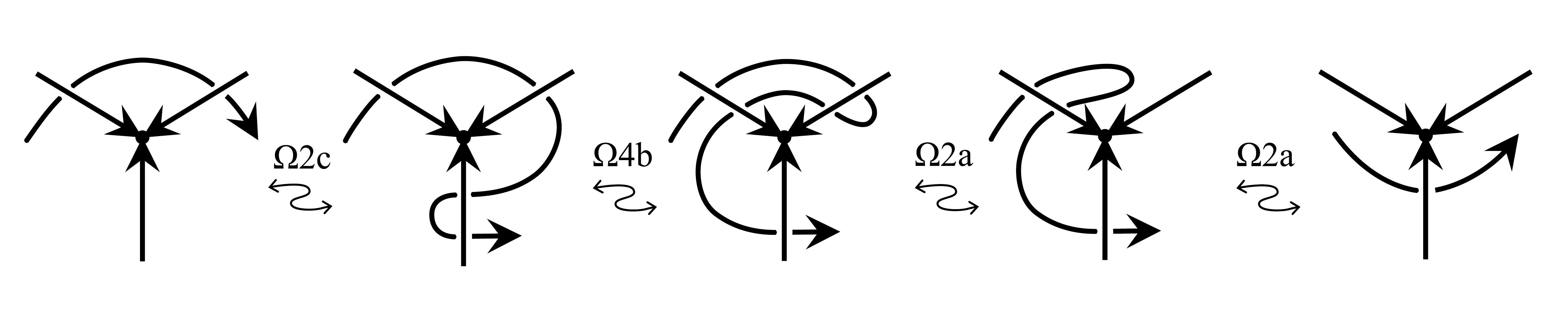}\\
\noindent\includegraphics[scale=0.2]{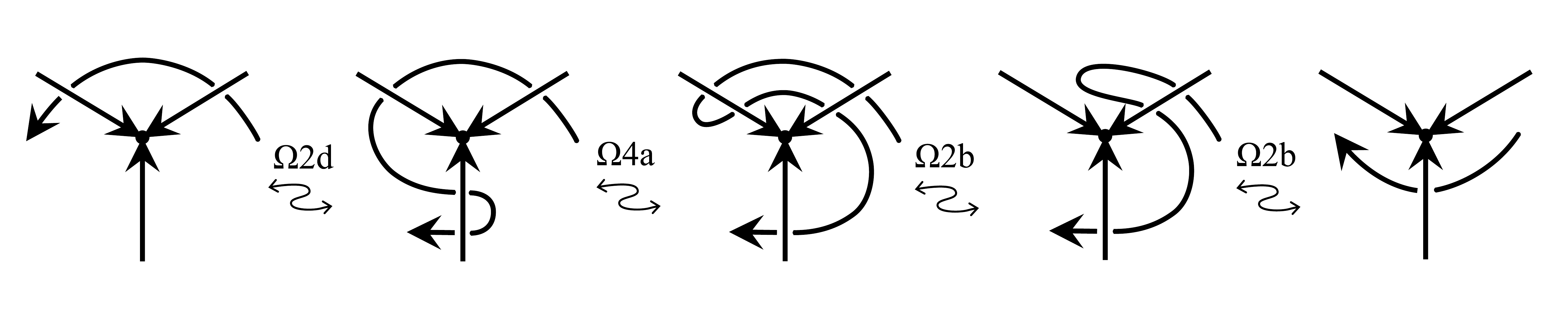}
\end{center}
Hence, only one of the moves $\Om4a$ or $\Om4b$ is necessary in a generating set of Reidemeister-type moves for oriented spatial trivalent graph diagrams. 

Similarly, the moves $\Om4c$ and $\Om4d$ allow passing a strand over a sink vertex. But as we prove below, only one of the moves $\Om4c$ or $\Om4d$ is necessary in a generating set of Reidemeister-type moves for oriented spatial trivalent graph diagram.
\begin{center}
\noindent\includegraphics[scale=0.2]{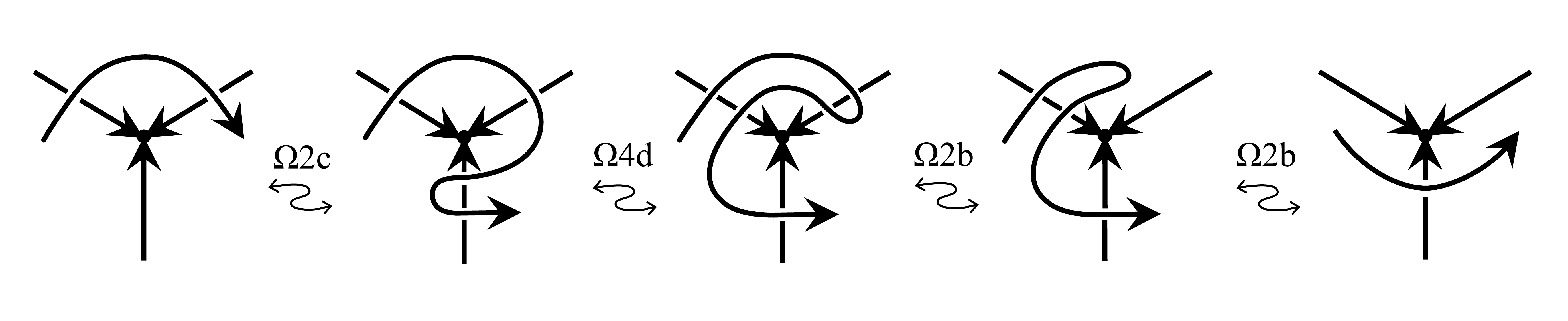}\\
\noindent\includegraphics[scale=0.2]{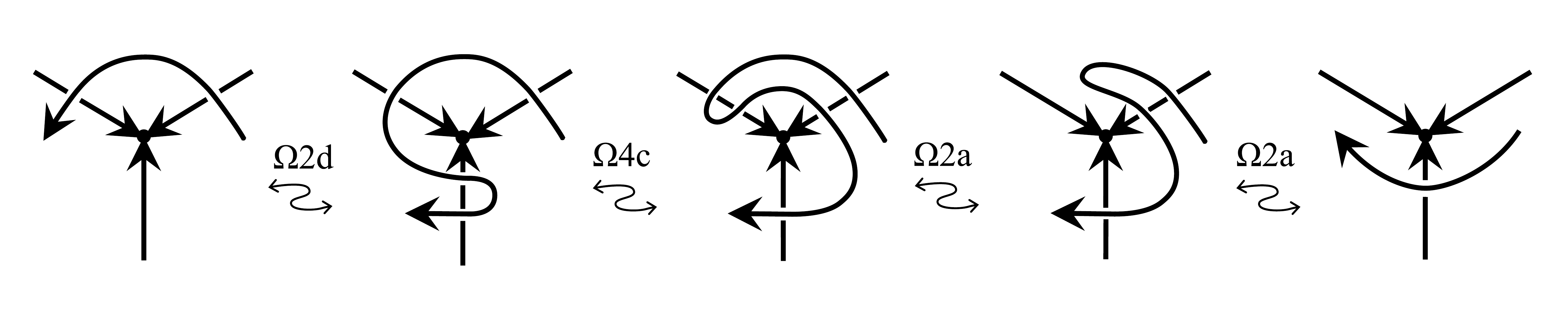}
\end{center}

The same situation applies to the moves $\Om4e$--$\Om4h$ involving a source vertex, and it can be shown in the same way that one of the moves $\Om4e$ or $\Om4f$ and also one of the moves $\Om4g$ or $\Om4h$ are needed in a generating set for oriented spatial trivalent graph diagrams.

Hence, a generating set of Reidemeister-type moves for oriented spatial trivalent graph diagrams with source and sink vertices contains four moves of type 4, with one move from each of the sets $\{\Om4a, \Om4b\}$, $\{\Om4c, \Om4d\}$, $\{\Om4e, \Om4f\}$, and $\{\Om4g, \Om4h\}$.
\end{proof}

\begin{theorem}\label{Type5Gen}
A minimal generating set of Reidemeister-type moves for oriented spatial trivalent graph diagrams with source and sink vertices contains two $\Om5$ moves, with exactly one move from each of the sets $\{\Om5a, \Om5b\}$ and $\{\Om5c, \Om5d\}$.
\end{theorem}

\begin{proof}
The moves $\Om5a$ and $\Om5b$ involve a sink vertex whereas the moves $\Om5c$ and $\Om5d$ involve a source vertex. Thus, a set that generates all oriented Reidemeister-type moves must include at least two oriented $\Om5$ moves, one that involves a sink vertex and another one that involves a source vertex. We prove that in fact  one move is sufficient for each type of vertex. Below we show that the move $\Om5c$ can be realized through a sequence of the moves 
$\Om5d, \Om1a$ and $\Om4f$.
\begin{center}
\includegraphics[scale=.2]{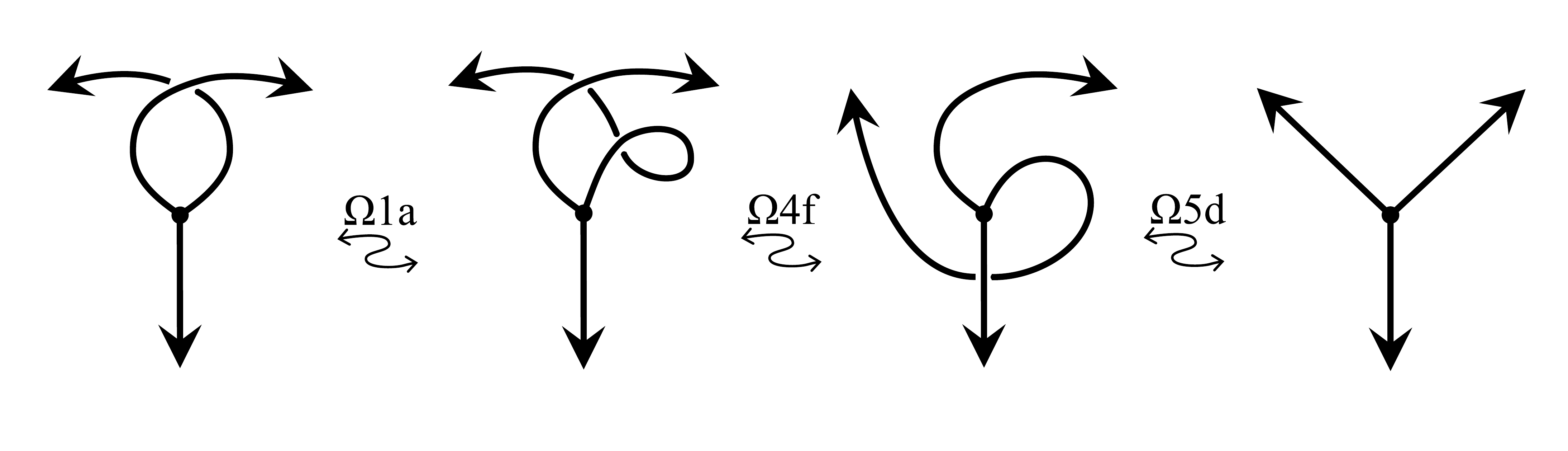}
\end{center}

Similarly, the move $\Om5d$ may be realized through a sequence of the moves $\Om5c, \Om1d$ and $\Om4e$, as shown below.
\begin{center}
\includegraphics[scale=.2]{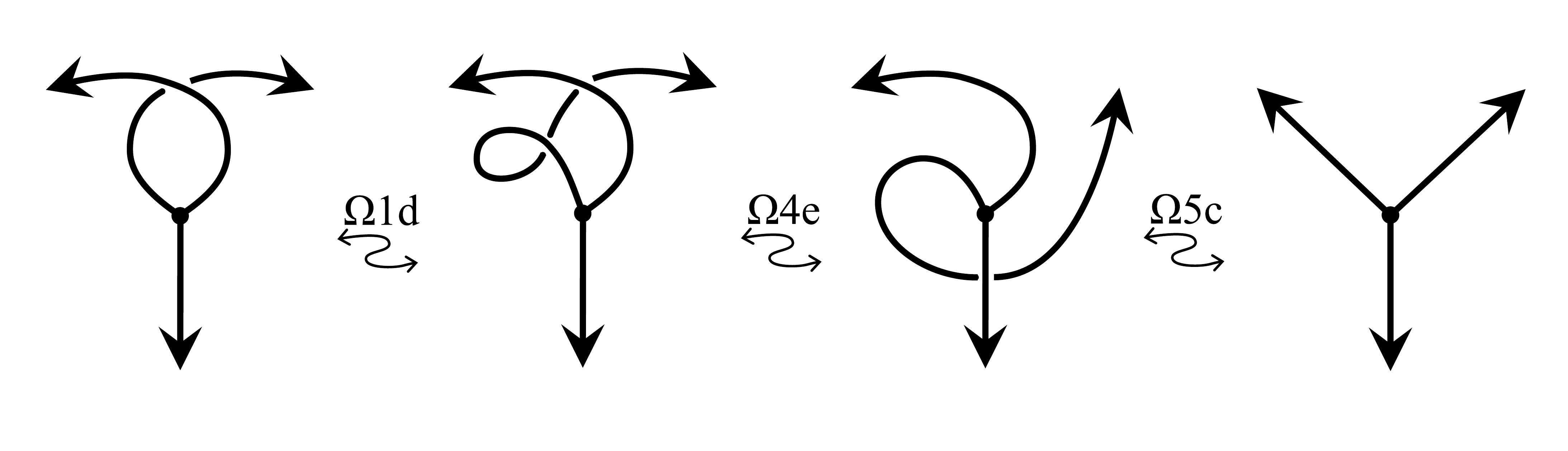}
\end{center}

The moves $\Om5a$ and $\Om5b$ involving a sink vertex can be used to realize each other. It can be shown in a similar way that the move $\Om5a$ can be realized through an application of the moves $\Om5b, \Om1b$ and $\Om4a$, and the move $\Om5b$ can be obtained via the moves  $\Om5a, \Om1c$ and $\Om4b$.
Therefore, the statement follows.
\end{proof}

As a result of Theorems \ref{Type4Gen} and \ref{Type5Gen}, combined with Theorem \ref{P1Converse}, a minimal generating set of oriented Reidemeister-type moves for spatial trivalent graph diagrams contains 10 moves: two $\Om1$ moves, one $\Om2$ move, one $\Om3$ move, four $\Om4$ moves, and two $\Om5$ moves. Specifically, the following follows.

\begin{corollary}
If $S$ is a minimal generating set of Reidemeister-type moves for oriented spatial trivalent graph diagrams with source and sink vertices, then $S$ contains 10 moves: four moves from any of the sets in the collection $\mathcal{A} \cup \mathcal{H}$ and one move from each of the  sets $\{\Om4a, \Om4b\}$, $\{\Om4c, \Om4d\}$, $\{\Om4e, \Om4f\}$, $\{\Om4g, \Om4h\}$, $\{\Om5a, \Om5b\}$ and $\{\Om5c, \Om5d\}$.
\end{corollary}

Since we have proven that there are exactly 12 minimal generating sets of classical Reidemeister moves for oriented knot diagrams, we can count the total number of minimal generating sets of Reidemeister-type moves for spatial trivalent graph diagrams with source and sink vertices. From Theorem \ref{Type4Gen}, there are $2^4 = 16$ choices of four $\Om4$ moves, and from Theorem \ref{Type5Gen} there are $2^2=4$ choices of two $\Om5$ moves. Hence, there are exactly $12\cdot16\cdot4=768$ minimal generating sets of oriented Reidemeister-type moves for spatial trivalent graph diagrams with source and sink vertices.

\end{document}